\documentclass{article}

\usepackage{amsmath, amsthm, amscd, amsfonts, amssymb, graphicx, color}
\usepackage[bookmarksnumbered, colorlinks, plainpages]{hyperref}
\usepackage{arxiv}

\usepackage[utf8]{inputenc} % allow utf-8 input
\usepackage[T1]{fontenc}    % use 8-bit T1 fonts
\usepackage{hyperref}       % hyperlinks
\usepackage{url}            % simple URL typesetting
\usepackage{booktabs}       % professional-quality tables
\usepackage{amsfonts}       % blackboard math symbols
\usepackage{nicefrac}       % compact symbols for 1/2, etc.
\usepackage{microtype}      % microtypography
\usepackage{lipsum}		% Can be removed after putting your text content
\usepackage{graphicx}
\usepackage{natbib}
\usepackage{doi}

\usepackage{lscape,xspace}
\usepackage{amscd,enumerate}
\usepackage{indentfirst}

\usepackage{amssymb}
\usepackage{amsfonts}
\usepackage{latexsym}

\usepackage{scalefnt}
\usepackage{float}
\usepackage{graphicx}
\usepackage{wasysym} 
\usepackage{ragged2e}
\usepackage{inputenc}
\usepackage{xcolor}

\newtheorem{theorem}{Theorem}[section]
\newtheorem{lemma}[theorem]{Lemma}

\theoremstyle{definition}
\newtheorem{definition}[theorem]{Definition}
\newtheorem{example}[theorem]{Example}

\theoremstyle{remark}

\numberwithin{equation}{section}

\newcommand{\SB}{\ensuremath{\mathrm{SB}}}
\newcommand{\PB}{\ensuremath{\mathrm{PB}}}

\title{The metric geometry of paper surfaces under geometric constraints}

\author{ \href{https://orcid.org/0000-0001-6187-2177}{\includegraphics[scale=0.06]{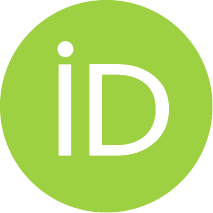}\hspace{1mm}Luciana  M.~Vasconcelos}\thanks{Corresponding author: lucianamvasc@ime.usp.br}\\
Institute of Mathematics and Statistics\\
	Department of Applied Mathematics\\
	 University of São Paulo\\
	São Paulo, SP, 05508-090, Brazil.\\
	\texttt{lucianamvasc@ime.usp.br} \\
}

\date{}

\renewcommand{\headeright}{}
\renewcommand{\undertitle}{}
\renewcommand{\shorttitle}{The metric geometry of paper surfaces}

\hypersetup{
pdftitle={The metric geometry of paper surfaces},
pdfauthor={Luciana M. Vasconcelos, André de Carvalho},
pdfkeywords={Paper surfaces, Ahlfors regularity, Linear local contractibility, Quasisymmetric maps},
}

\begin{document}
\maketitle

\begin{abstract}
	We investigate the quasisymmetric uniformization of a special class of metric surfaces known as paper surfaces, constructed as quotients of planar multipolygons via segment pairings, including infinite Type W identifications. These spaces, which arise naturally in dynamical settings, exhibit conic singularities and complex geometric structure. Our goal is to prove that a broad class of such surfaces satisfies Ahlfors 2-regularity and linear local contractibility, which together ensure the existence of a quasisymmetric parametrization onto the standard 2-sphere.
\end{abstract}

% keywords can be removed
\keywords{Paper surfaces \and Ahlfors regularity \and Linear local contractibility \and Quasisymmetric maps.}

\section{Introduction} 

The problem of uniformization has evolved over time, moving from classical conformal geometry to modern approaches involving metric geometry and quasisymmetric mappings. This field has grown significantly, especially with the study of fractals, hyperbolic groups, and computational tools for geometric structures.

In 1851, Bernhard Riemann showed that any simply connected open set in the complex plane, except the whole plane, can be conformally mapped to the unit disk. This laid the foundation for modern conformal mapping. Later, the uniformization theorem extended this idea, stating that every simply connected Riemann surface is conformally equivalent to the Riemann sphere, the complex plane, or the unit disk.

The field later expanded as researchers applied these concepts to more general settings, developing a rich theory of geometric analysis. The focus shifted to constructing parameterizations, or uniformizations, for metric spaces with specific geometric properties, aiming to classify spaces with similar structures.

In 1954, Lars Ahlfors \cite{ahlfors1954} advanced the field by introducing quasiconformal mappings and conformal invariants, which laid the groundwork for later work on quasisymmetric uniformization. Significant contributions followed, including the results of Semmes \cite{semmes1991chordarc}, David and Semmes \cite{david1993quantitative}, and, in 2002, Mario Bonk and Bruce Kleiner \cite{bonk2002quasisymmetric}, who established a metric version of the uniformization theorem. They showed that any Ahlfors 2-regular, linearly locally contractible metric sphere can be quasisymmetrically mapped to the topological 2-sphere. Subsequent studies \cite{bonk2005conformal, merenkov2013quasisymmetric, wildrick2008quasisymmetric, wildrick2010quasisymmetric} further developed these results by exploring additional geometric conditions. There the Ahlfors regularity condition is combined with varying
geometric conditions on the codomain.

Recent work has examined conditions under which a space can be quasisymmetrically mapped to the topological 2-sphere. Researchers have explored various types of metric surfaces and domains with boundaries. For example, Fitzi and Meier \cite{fitzi2022canonical} proved that surfaces with specific properties can be mapped to model surfaces with the same topology.

We now turn to the problem of paper space uniformization. The study of paper spaces extends these ideas further. Paper spaces are constructed by identifying boundary segments of polygons, resulting in surfaces that are flat except for singular points. De Carvalho and Hall \cite{de2012paper} introduced paper-folding schemes, connecting these constructions with dynamical systems. Their work showed how specific folding patterns yield unique conformal structures.

%André de Carvalho and Toby Hall \cite{de2012paper} introduced the concept of paper-folding schemes, which connect these geometric constructions with dynamical systems. Their research revealed how certain folding patterns lead to unique conformal structures on these surfaces.

For instance, when a genus-two surface is formed by identifying opposite edges of a regular octagon, all eight vertices collapse into a single cone point with a total angle of $6\pi$. More generally, any closed surface of genus $g \geq 1$ can be topologically obtained by identifying the edges of a regular $4g$-gon.

%This construction induces a natural conformal structure. Away from singularities, the structure follows the polygon's Euclidean geometry; at cone points, it can still be described via fractional power coordinate transformations.

%Returning to the general case of paper spaces, there are numerous examples of such structures in mathematics. One of the most familiar cases is the torus, which can be represented as the quotient of a square by identifying pairs of opposite sides. More generally, any closed surface of genus $g\ge 1$ can be topologically constructed as a paper space by identifying the edges of a regular  $4g$-gon appropriately. Other significant examples of paper spaces, particularly relevant in the study of billiards theory, include the so-called translation surfaces, as discussed in \cite{Zor06} and \cite{Ran16}.

%However, the study of paper spaces is interesting in its own right, as their definition includes quotient polygons from other contexts, such as infinite-type translation surfaces, which also hold relevance in Dynamical Systems (see \cite{clavier2019rotational} and references therein).

Paper spaces also have applications beyond pure mathematics. In dynamical systems,  particularly relevant in the study of billiards theory, include the so-called translation surfaces, as discussed in \cite{Zor06} and \cite{Ran16}. Their definition includes quotient polygons from other contexts, such as infinite-type translation surfaces, which also hold relevance in Dynamical Systems %(see 
\cite{clavier2019rotational}. 
%and references therein). 
De Carvalho and Hall \cite{de2014riemann} provided conditions under which the complex structure within a polygon extends uniquely across boundary quotients, producing a closed Riemann surface.

This construction induces a natural conformal structure on the surface in addition to its metric structure. Away from singularities, the conformal structure is directly inherited from the polygon's Euclidean geometry. However, at cone points, this structure can still be recovered by applying fractional power coordinate transformations, allowing for well-defined local holomorphic charts.

%This idea can be further generalized by considering polygons where infinitely many pairs of boundary segments are identified, broadening the class of surfaces that can be studied through this method.  

%In their paper \cite{de2014riemann} André de Carvalho and Toby Hall provide a condition under which the complex structure within the polygon extends uniquely across the boundary's quotient, resulting in a closed Riemann surface. Additionally, they establish a modulus of continuity for a uniformizing map on such surfaces.

%An important remark is that many paper spaces with dynamic interest do not satisfy the bounded curvature condition as defined in \cite{BBI}, which makes it impossible to directly apply existing results to their analysis.

However, challenges remain. Many dynamically relevant paper spaces fail to meet the bounded curvature condition from \cite{BBI}, preventing the direct application of existing results. Paper-folding schemes can involve infinitely many identified boundary segments, complicating the uniformization analysis.

An important case is the so-called tight horseshoe. This space does not satisfy the Ahlfors regularity condition, preventing the application of Bonk-Kleiner's theorem to guarantee a quasisymmetric map to the 2-sphere. However, it is known that the underlying complex structure still extends uniquely to this point, providing a conformal uniformizing map \cite{de2012paper}.

In dynamical systems, the constructions of paper spaces have proven useful in describing pseudo-Anosov transformations \cite{fathi1979travaux}. Paper spaces, while initially studied for their dynamic properties, have since gained recognition in metric geometry due to their structural complexity and potential for uniformization.

A paper-folding scheme can be formally described as an equivalence relation that identifies boundary segments, possibly infinitely many, of a given finite collection of disjoint polygons in the complex plane. The image of the polygon boundary in the resulting quotient space $ S $ is called the scar. This structure contains cone points, where the total angle differs from $ 2\pi $, and singular points, which can appear as accumulation points of cone singularities. These features play a crucial role in understanding the geometric and topological properties of the space, affecting whether it can be quasisymmetrically uniformized.

Quasisymmetric mappings play a key role here. A homeomorphism $f: X \to Y$ is $\eta$-quasisymmetric if there exists a function $\eta: [0,\infty) \to [0,\infty)$ such that

$$
    d_X(x,a) \leq t d_X(x,b) \implies d_Y(f(x),f(a)) \leq \eta(t) d_Y(f(x),f(b))
$$
for all $a, b, x \in X$ and $t > 0$. Such mappings extend conformal ideas to metric spaces without differential structures. 
These mappings were initially explored in metric spaces by Tukia and Väisälä \cite{tukia1980quasisymmetric}, building on the foundational work of Beurling and Ahlfors \cite{beurling1956boundary}, who studied the boundary behavior of quasiconformal maps in two dimensions. Tukia and Väisälä established key properties of quasisymmetric mappings and presented the first uniformization result involving these mappings: a complete characterization of quasisymmetric circles based on intrinsic properties of the metric space.

Quasisymmetric mappings also preserve the doubling and LLC (linearly locally connected) properties, which makes them a natural starting point for uniformization problems. Many well-known spaces, such as Euclidean spaces, balls, and spheres, exhibit these properties. The theory of quasisymmetric mappings between metric spaces has advanced significantly following the foundational work of Tukia and Väisälä.

This research continues the investigation of paper spaces from the perspective of metric geometry, following the principles discussed in \cite{BBI} and \cite{aleksandrov1967intrinsic} and investigates the quasisymmetric uniformization of paper surfaces and presents conditions under which these surfaces can be mapped to the topological 2-sphere. We prove the following main result: 

%espaços de papel são importantes em dinâmica, por um lado, e são espaços métricos especiais aos quais se aplica essa teoria de geometria métrica, por outro. No artigo, prova que certos espaços de papel, com identificações motivadas por dinâmica, são Ahlfors 2-regulares e llc e conclui, pelo teorema \cite{bonk2002quasisymmetric}, que são qs à esfera see sec \ref{sec-surfL}. é important ressaltar queue as tecnicas utilizadas nesse artigo são inconclusive em alguns esc=paços de papel como por exemplo an chamada ferradura justa queue não satisfaz a condição de ser ahfors regular e consequent mente não consequimos utilizar o teorema de Bon-Kleiner para assegurar queue é qs a 2 esfera. Até entao o queue podemos afircar sobre espaços como a ferradura justa é qua  a estrutura complexa subjacente se estende unicamente sobre aquele ponto e, portanto, há uma transformação de uniformização conforme para a esfera \cite{de2012paper}.

%This paper investigates when paper surfaces can be quasisymmetrically uniformized to the topological 2-sphere. It presents conditions that make this possible, highlighting the connection between the geometric properties of the folding schemes and the resulting surfaces.

%Our main goal with this paper is to prove the following:

\medskip

\textbf{Theorem \ref{teo_principal}.}
\textit{The surfaces in the class $ \mathcal{L} $ - paper surfaces constructed from a multipolygon with Type~$W$ identifications or basic segment pairings (see Section \ref{sec-surfL}) - satisfy the following properties:
\begin{itemize}
    \item[(a)] Ahlfors 2-regularity;
    \item[(b)] Linear local contractibility.
\end{itemize}
As a consequence, the surfaces of type $ \mathcal{L}^* $ are quasisymmetrically equivalent to the 2-sphere.}

\medskip

This theorem proves that even surfaces with singularities, like cone points or singular accumulations, can be mapped to the 2-sphere. It also outlines conditions under which such mappings are not possible.

The types of pairings are motivated by dynamics: surfaces associated with unimodal generalized pseudo-Anosov maps can be viewed as metric quotients of certain polygons. In this context, the identifications along the edges of the square are of Type~W identifications, with a distinguished point where the associated sequences decay geometrically.

%This theorem shows that even surfaces with singularities, like conical points or singular accumulations, can be quasisymmetrically mapped to the 2-sphere. It also clarifies when this mapping is not possible, defining the limits of such transformations.

%Understanding the quasisymmetric uniformization of paper surfaces contributes to a better grasp of metric space geometry. The connections between geometric structures and quasisymmetric mappings open doors to further research, with potential applications in conformal geometry, dynamical systems, and mathematical models of paper-folded surfaces.

%Understanding the quasisymmetric uniformization of paper surfaces helps us better understand the geometry of these spaces and how they relate to dynamical systems. For example, in Section \ref{application}, we provide an example of a surface associated with a generalized pseudo-Anosov application, which can be seen as a paper surface, as stated in the previous Theorem.

Therefore, a comprehensive understanding of the quasisymmetric uniformization of paper surfaces yields significant insights into the underlying geometric structures of these spaces and their interplay with dynamical systems. In particular, as demonstrated in the preceding theorem, the surface arising from a generalized pseudo-Anosov map may be interpreted as a paper surface within this framework.

%We also explore extensions to cases where certain rectangular regions are removed, showing that the resulting surfaces maintain Ahlfors regularity and LLC properties, ensuring their continued quasisymmetric equivalence to the 2-sphere (see Section \ref{seclinguetas}).

%Uma pergunta que pode ser feita é see é possivel conculir essa quase simetria para outros espaços de opapel sobre o quadrado mas removendo algumas partes dele  e a resposta é positive. nesse sentido, nesse artigo provaremos queue  ao removermos uma quantidade finita retangulos, com certos tipos de identificações, ver sec \ref{seclinguetas} dos espacos contemplados pelo teorema principal épossivel ainda garantir queue else continual sendo ahfors regular e LLC e consequentemente quase simetricamente equivalent a 2 esfera. 

\section{Background}

This section presents an overview of key theoretical concepts and results that serve as the foundation for the paper, along with references for more detailed discussions.

Section~\ref{sec:quotient_metric} introduces the theory of quotient metric spaces and explores their relationship with topological quotients, establishing the necessary background for subsequent sections. %Section~\ref{sec:geom_not} covers fundamental topological results, including the universal property of quotient topology. 
Additionally, it is shown that a homotopy respecting a given relation induces a homotopy in the quotient space.%, a result that plays a crucial role in Section~\ref{sec: applications}. 

The notion of \textit{paper spaces}, which constitute the central objects of this work, is discussed in Section~\ref{sec:paper_spaces}, along with their fundamental properties and preliminary results. Section~\ref{sec:metric_structure} focuses on the metric structure of these spaces, proving that every paper space can be regarded as a conic flat surface. Finally, Section~\ref{sec:qs_unif} examines metric spaces that are homeomorphic to the sphere, analyzing the conditions under which such spaces can be quasisymmetrically mapped onto the standard Euclidean sphere~$S^2$.

\subsection{Metric spaces: quotient spaces and intrinsic metric }\label{sec:quotient_metric}

In this section, some conventions to be used throughout the text will be introduced, along with a review of fundamental definitions related to metrics. The concepts discussed in this section can be found in greater detail in \cite{BBI} and \cite{de2012paper}, which serves as an excellent reference for further exploration.

A \textit{metric} on a set $ X $ is a function $ d_X : X \times X \to \mathbb{R} \cup \{\infty\} $ that satisfies the usual conditions: it is non-negative, symmetric, and satisfies the triangle inequality. Additionally, the distance between two distinct points is always positive, and the distance is zero if and only if the points are identical.
Given this, a \textit{metric space} consists of a set equipped with a metric. Formally, it is defined as a pair $ (X, d_X) $, where $ d_X $ is a metric on $ X $. The elements of $ X $ are referred to as \textit{points of the metric space}, and $ d_X(x, y) $ represents the \textit{distance} between the points $ x $ and $ y $.

A function $ d_X : X \times X \to \mathbb{R}_+ \cup \{+\infty\} $ is called a \textit{semimetric} or \textit{pseudometric} if it satisfies all the properties of a metric, except that $ d_X(x, y) = 0 $ need not imply $ x = y $, that is, distinct points are allowed to have distance zero. If $ d_X $ is a pseudometric, then $ d_X(x,y) $ represents the \textit{pseudodistance} between the points $ x $ and $ y $. There's a canonical way of turning a pseudometric space $\left(X,d_X\right)$ into a metric space $\left(X/{d_R},d_X\right)$: declare two points to be equivalent if the pseudodistance between them is zero and set the distance between equivalence classes $[x],[y]$ to be the distance between any two elements in the classes, that is, $d_R([x],[y]):=d_R(x,y)$, where the first $d_R$ is the distance on $X/{\sim_R}$ and the second is the pseudometric on $X$.

\textit{Notation. }For a metric or pseudometric space $ (X, d_X) $, the following notation is used. Given a point $ x \in X $ and a radius $ r \geq 0 $, the open and closed balls centered at $ x $ with radius $ r $ are defined as:
$$
B_X(x, r) := \{ y \in X \mid d_X(y, x) < r \},
$$
$$
\overline{B}_X(x, r) := \{ y \in X \mid d_X(y, x) \leq r \}.
$$
When the ambient space is clear from the context, the simplified notation $ B(x, r) $ may be used.

Let $(X, d)$ be a metric space. A \textit{path }in $X$ is defined as a continuous map $\gamma: I\rightarrow X$ where $I$ is the interval $[0,1]$ or $[a,b]\subset \mathbb{R}$. The \textit{length} of $\gamma: [a, b] \to X$ is given by
$$
|\gamma|_X = \sup \left\{ \sum d(\gamma(t_i), \gamma(t_{i+1})) \right\} \in \mathbb{R}_{\geq 0} \cup \{\infty\},
$$
where the supremum is taken over all finite partitions $a = t_0 < t_1 < \cdots < t_k = b$ of the interval $[a, b]$.

This concept of path length leads to the notion of an intrinsic metric. Specifically, a metric is \textit{intrinsic} if the distance between any two points can be arbitrarily well-approximated by the lengths of curves joining the two points. The metric is \textit{strictly intrinsic} if the infimum is attained. That is, for every $x, y \in X$, there exists a path from $x$ to $y$ whose length equals $d(x, y)$. In particular, $d(x, y) = \infty$ if there is no path from $x$ to $y$. It can be shown that a compact intrinsic metric is strictly intrinsic.

If $d$ is not an intrinsic metric, then there is an \textit{induced intrinsic metric} $\hat{d}$ on $X$, defined by
$$
\hat{d}(x, y) = \inf \{ |\gamma|_X  \ | \  \gamma: [a, b] \to X \text{ is a path with } \gamma(a) = x, \, \gamma(b) = y \}.
$$

For example, if $P \subset \mathbb{R}^2$, there is an intrinsic metric $d_P$ on $P$ induced by the Euclidean metric on $\mathbb{R}^2$. However, this metric need not agree with the subspace metric on $P$.

%The following theorem provides a key result to ensure that the metric of a paper space is strictly intrinsic. A more general statement of this result can be found in \cite{BBI}, Theorem 2.5.23.

%%\begin{thm}\label{teo1-1-11}
 %   Every compact intrinsic metric is strictly intrinsic.
%\end{thm}

\noindent

%The following theorem establishes essential properties of intrinsic metrics and length spaces, ensuring that certain metric structures remain consistent under quotient operations. 

The key properties of intrinsic metrics and length spaces ensure that certain metric structures remain consistent under quotient operations. Specifically, one such property is that the metric quotient of a length space is itself a length space.

In simple terms, a \textit{length space} is a metric space whose metric is intrinsic, meaning the distance between any two points can be realized as the length of a suitable path connecting them. More formally, in a length space, the distance function is defined as the infimum of the lengths of all possible paths between two points. This concept ensures that geodesics (shortest paths) can be found within the space, making it a central object of study in metric geometry. %Further details and a comprehensive discussion can be found in \cite{BBI}.

    Let $(X,d)$ be a metric space and $R$ be a symmetric and reflexive relation on $X$. A $R$\textit{-chain} from $x$ to $y$ is a sequence $((p_i,q_i))^k_{i=0}$ in $X^2$ such that $xRp_0, q_iRp_ {i+1}$ for $i = 0 , 1, ..., k - 1$, and $q_kRy$. The \textit{length of the $R$-chain} is given by $$L^R((p_i,q_i)) := \sum^k_{i=0}d(p_i,q_i).$$

%This definition states that to get from $x$ to $y$ with a $R$-chain, it is better to move between sets in $R$ but not within them. It is possible to 
The relation $R$ can be seen as a collection of subsets of $X$ where the relation $xRy$ holds if $x$ and $y$ are in the same set in the collection $R$.

 Collections of subsets can always be associated with corresponding relations.  Given a collection $\mathcal{G}$ of subsets of a metric space $X$, it follows that it will determine a reflective and symmetric relation $R_\mathcal{G}$ on $X$ defined as follows:
$$x R_\mathcal{G} y\Leftrightarrow x \ \text{and} \ y \ \text{are} \ \mathcal{G}  \text{-equivalent},$$
meaning either $x=y$ or there exists  $g\in \mathcal{G}$ with $x,y\in g$.

The function $ d^R : X \times X \rightarrow \mathbb{R}_+ \cup \{+\infty\} $ is defined, for each $ x $ and $ y $ in $ X $, by
$$
d^R(x,y) := \inf\{L^R((p_i,q_i)):((p_i,q_i)) \ \mbox{is an }R\mbox{-chain of} \ x \ \text {to} \ y \}.
$$
Notice that $ d^R $ is a \textit{pseudometric} on $ X $ and is well-defined since there always exists at least one $ R $-chain from $ x $ to $ y $, namely the trivial one.

The equivalence relation on $X$ that identifies points that have $ d^R $-pseudodistance zero is denoted $ \sim_R $, and the quotient space under this equivalence relation is the \emph{metric quotient space} of $ (X, d) $ under the relation $ R $, denoted either $(X / d^R, d^R)$ or $(X/{\sim_R},d^R)$. The distance between equivalence classes $[x], [y]\in X/{\sim_R}$ is defined by $d^R([x],[y]):=d^R(x,y)$ and it is easy to see that this is indeed a metric. 

Many mathematicians tend to be more familiar with the concept of topological quotients. Given an equivalence relation $ R $, there are two primary reasons why the metric quotient space $ (X/d^R, d^R) $ might not coincide with the topological quotient $ X/R $. 
The first concerns their underlying sets: for instance, the metric quotient of $ \mathbb{R} $ by $ \mathbb{Q} $ reduces to a single point, since there are arbitrarily short $ R $-chains joining any two elements.

The second difference lies in their topologies: even if the quotients agree as sets, they may carry distinct topologies. A representative example (see Example 3.1.17 in~\cite{BBI}) involves a space $ X $ constructed as the disjoint union of countably many intervals $ I_i $, each with length $ \ell_i $, and a relation identifying all the left endpoints. The metric quotient’s topology depends on the sequence $ \ell_i $: it is compact if $ \ell_i \to 0 $, and non-compact otherwise. In contrast, the topological quotient is unaffected by the particular values of $ \ell_i $.

The problem of having quotients which agree as sets but not as topological spaces does not arise
if $X$ is compact (see \cite{BBI}): If 
$(X,d)$ is a compact metric space and $R$ is a reflexive and symmetric relation on $X$, then the metric quotient $(X/d^R,d^R)$ and the topological quotient $X/ R$ are homeomorphic. Furthermore, if $(X,d)$ is a length space and $R$ is a reflexive and symmetric relation in $X$, then the metric quotient $(X/d^R,d^R)$ is also a length space.

\noindent

    A \textit{separation} of a topological space $X$ is a decomposition of $X$ as a disjoint union $X = A \cup B$ where $A$ and $B$ are non-empty closed subsets of $X$. We say that $X$ is \textit{connected} if there is no separation of X. A subset $C \subset X$ \textit{separates} two points $x,y \in X$ (or two subsets $D,E \subset X)$ if there is a separation $X \setminus C = A \cup B$, with $x \in A$ and y $\in B \ (\text{or } D \subset A$ and $E \subset B$). If X is connected and $X \setminus C$ is not, we say that $C$ \textit{separates} X. %A \textit{continuum} is a compact and connected Hausdorff space.

  An equivalence relation $R$ in a topological space $X$ is said to be \textit{closed} if $R$ is closed as a subset in the product space $X \times X$.

        Furthermore, a \textit{partition or decomposition} $\mathcal{G}$ of a topological space into compact sets is \textit{upper semi-continuous} if, for every element $\zeta \in \mathcal{G}$ and every open set $U\supset \zeta$, there exists an open $V\subset U$ with $\zeta \subset V$ such that every $\zeta' \subset \mathcal{G}$ with $\zeta'\cap V\neq \emptyset$,  we have $ \zeta' \subset U$. The decomposition is \textit{monotone } if its elements are connected.

An important theorem by R.L. Moore provides sufficient conditions for the topological quotient of a 2-sphere to be a 2-sphere again. %This theorem forms the basis for the following fact: every simple folding produces a paper space homeomorphic to the two-dimensional sphere.

\begin{theorem}[\cite{moore1925upper}] Let $\mathcal{G}$ be an upper semi-continuous and monotone decomposition of a topological sphere $S$ such that no element of $\mathcal{G}$ separates $S$. Then, the quotient space obtained by collapsing each decomposition element to a point is a topological sphere.
\end{theorem}

To connect homotopies with quotient spaces, consider a set $X$ with an equivalence relation $R$. A homotopy $\tilde{H}:X \times I \to X$ that respects $R$ pointwise along $t$ induces a well-defined homotopy on the quotient space $X/R$. This relies on the fact that any continuous map on $X$ preserving the equivalence relation extends uniquely to $X/R$, ensuring that the behavior of functions is consistently preserved under the quotient operation. The following theorem formalizes this construction.

%To relate homotopies with quotient spaces, consider a set $X$ equipped with an equivalence relation $R$. If a homotopy $\tilde{H}:X \times I \to X$ respects the relation $R$ pointwise along the parameter $t$, it induces a well-defined homotopy on the quotient space $X/R$. The following theorem formalizes this construction.

%The theorem guarantees that any continuous function defined on the original space that respects the equivalence relation can be uniquely extended to a continuous function on the quotient space. This is essential for ensuring that the properties of functions are preserved under the quotient operation.

\begin{theorem}\label{univ.thm}
	Given $ X $ a set and an equivalence relation $ R $ on $ X $, let $ \tilde{H}: X \times I \to X $ be a homotopy in $ X $ such that for all $ x, y \in X $, if $ x $ and $ y $ are equivalent under the relation $ R $, then $ \tilde{H}(x, t) = \tilde{H}(y, t) $ for all $ t \in I $, where $ I $ is the unit interval. There exists a unique homotopy $ F: X/R \times I \to I/R $ in the quotient such that 
    $$
    \pi_R \circ \tilde{H}_t = F_t \circ \pi_R.
    $$%\textcolor{red}{\tt Everything is ok with this theorem?}
\end{theorem}

\begin{proof}
    
For each $t\in I$, by the Universal Property of the Quotient Topology, there exist a unique continuous function $f_t:X/R\to X$ such that 
$$\Tilde{H}_t=f_t\circ \pi_R.$$

Therefore, 
$$\pi_R\circ \Tilde{H}_t=\pi_R \circ f_t \circ \pi_R.$$

Defining $F_t:=\pi_R \circ f_t$, the result follows. 
    
\end{proof}

\subsection{Paper Surfaces} \label{sec:paper_spaces}

In this section, paper spaces and their properties are examined. These spaces arise as quotients formed by identifying pairs of sub-segments from the boundary of a flat polygon or, more generally, a flat multipolygon. The resulting structure is orientable and exhibits a flat conic geometry. A well-known example is the torus, which can be obtained as the quotient of a square by identifying pairs of opposite sides. The construction is not limited to a finite number of identifications; surfaces can also be formed by identifying an infinite collection of segments along the edges of $ P $.
The primary references for this section include \cite{de2010paper}, \cite{de2012paper}, \cite{figueroa2018estimativas}, and \cite{Marceldisserta}. 

We start by fixing some notation and presenting some definitions. An \textit{arc} in a metric space $ X $ is a continuous and injective map $ \gamma: [0,1] \to X $, such that $ \gamma([0,1]) \subset X $. The points $ \gamma(0) $ and $ \gamma(1) $ are referred to as its \textit{endpoints}, while the \textit{interior} of the arc, denoted $ \mathring{\gamma} $, corresponds to $ \gamma((0,1)) $. 

A \textit{segment} is a special case of an arc in $ \mathbb{C} $ that lies entirely within a straight line. The length of a segment $ \alpha $ is represented by $ |\alpha| $.

A \textit{simple closed curve} in $ X $ is the image of a continuous injective function from the unit circle into $ X $. When an arc or a simple closed curve in $ \mathbb{C} $ consists of a finite sequence of connected line segments, it is termed \textit{polygonal}. The longest contiguous segments forming a polygonal curve are called its \textit{edges}, and their endpoints are known as \textit{vertices}. 

A \textit{polygon} is a closed topological disk in $ \mathbb{C} $ whose boundary forms a polygonal simple closed curve. Its \textit{vertices} correspond to those of its boundary, and its \textit{sides} are the edges forming its perimeter.

A \textit{multipolygon} is a disjoint union of a finite number of polygons. Similarly, a \textit{polygonal multicurve} consists of a disjoint union of finitely many polygonal simple closed curves.

Consider $P \subset \mathbb{R}^2$ to be a multipolygon and let $C = \partial P$ denote its boundary, with each component of $C$ oriented positively.

\begin{definition}
  Let $\alpha, \alpha' \subset C$ be closed, non-trivial segments of the same length with disjoint interiors. The \textit{segment pairing} $\left\langle \alpha,\alpha' \right\rangle$ is the equivalence relation that identifies pairs of points on $\alpha$ and $\alpha'$ in a length-preserving and orientation-reversing manner. Specifically, if $\gamma_1:[0,L]\rightarrow\mathbb{R}^2$ and $\gamma_2:[0,L]\rightarrow \mathbb{R}^2$ are parameterizations of $\alpha$ and $\alpha'$, respectively, by arc length $(\text{where  }L=|\alpha|=|\alpha'|)$ and are compatible with the orientation of $C$, then  each element of $\langle \alpha, \alpha' \rangle$ is of the form $\{\gamma_1(t),\gamma_2(L-t)\}$. The segments $\alpha$ and $\alpha'$ need not coincide with entire edges of $C$ and, in 
most cases, they do not.
\end{definition}

\begin{figure}[!h]
    \centering
    \includegraphics[width=0.45\linewidth]{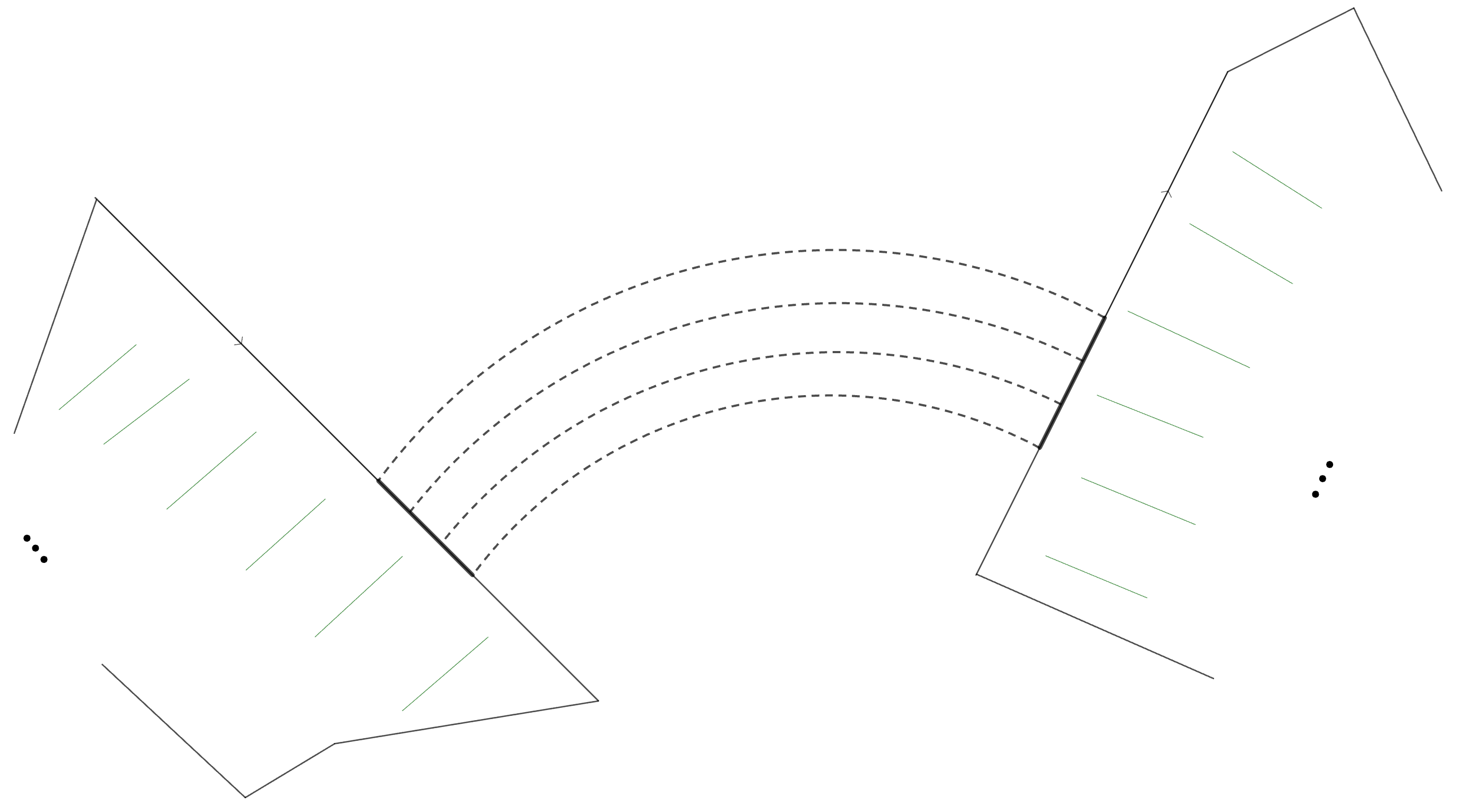}
    \caption{Segment pairing.}
    \label{fig:seg_pairing}
\end{figure}

The segments $\alpha, \alpha'$ and any two points that are identified under the pairing are said to be \textit{paired}. Two paired points that lie in the interior of a segment pairing form an \textit{interior pair}. Note that $\left\langle \alpha,\alpha' \right\rangle$ and $\left\langle \alpha',\alpha \right\rangle$ represent the same pairing.

Two segment pairings $\left\langle \alpha, \alpha' \right\rangle$ and $\left\langle \beta, \beta' \right\rangle$ have \textit{disjoint interiors} if the interiors of all four segments are disjoint.

The \textit{length} of a paired segment $\left\langle \alpha, \alpha' \right\rangle$, denoted $|\left\langle \alpha, \alpha' \right\rangle|$, is the length of one of the segments in the pairing, i.e. $$|\left\langle \alpha, \alpha' \right\rangle| := |\alpha|_{\mathbb{R}^2} = |\alpha'|_{\mathbb{R}^2}.$$

If $\mathcal{P} = {\left\langle \alpha_i, \alpha_i' \right\rangle}$ is a (countable) collection of pairwise interior-disjoint segment pairings, its length, denoted  $|\mathcal {P}|$, is the sum of the lengths of the pairings in $\mathcal{P}$, that is, $$|\mathcal{P}| = \sum_i |\left\langle \alpha_i, \alpha_i' \right\rangle|.$$

A collection $\mathcal{P}$ of interior-disjoints pairings is \textit{full} if $|\mathcal{P}|$ equals half the length of $C$. This means that the pairings in $\mathcal{P}$ cover all of $C$ up to a set of Lebesgue 1-dimensional measure zero.

A pairing of two segments that have one endpoint in common is called a \textit{fold}, and the common endpoint is referred to as the  \textit{folding point}. %Thus, the folding head in a fold is alone in its equivalence class, since the elements of a pairing, and therefore of a collection $\mathcal{P}$ as defined earlier, have at most two elements, having only one if and only if that element is a endpoint.

An interior-disjoint collection $ \mathcal{P} $ of segment pairings induces a reflexive and symmetric \textit{pairing relation}, also denoted by $ \mathcal{P} $, such that  
$$
\mathcal{P} = \{(x, x') \mid x, x' \text{ are paired or } x = x'\}.
$$

\begin{definition}
 
  A \textit{paper folding scheme}, or simply a \textit{paper folding}, is a pair $(P, \mathcal{P})$ where $P \subset \mathbb{R}^2$ is a multipolygon with the intrinsic metric $d_P$ induced from $\mathbb{R}^2$, and $\mathcal{P}$ is a full collection of interior-disjoint segment pairings in $\partial P$. The metric quotient $S := P / d_P^\mathcal{P}$ of $P$ under the pseudometric $d_S=d_P^{\mathcal{P}}$ induced by the pairing relation $$R_\mathcal{ P} := \{(x,x') \ | \ x = x' \mbox{or }x,x' \mbox{are paired under some pairing in } \mathcal{P}\}$$ is the  associated \textit{paper space}. If it is known that $S$ is a closed (compact without boundary) topological surface, then $(P,\mathcal{P})$ is a \textit{surface paper-folding scheme} and $S$ is the associated \textit{paper surface}.
\end{definition}

The projection map is denoted by $ \pi\colon P \to S $, and the quotient $ G = \pi(\partial P) \subset S $ of the boundary is referred to as the \textit{scar}. Notice that the restriction $ \pi: \operatorname{Int}(P) \to S \setminus G $ is a homeomorphism. When referring to points of $ P $ and $ S $ associated via $ \pi $, symbols with and without a tilde will often be used. For instance, if $ x \in S $, then $ \tilde{x} \in P $ will denote an arbitrary element of $ \pi^{-1}(x) $.

%The focus will be on plain folding schemes, which we define below, as they represent both the most fundamental and frequently encountered type of paper folding. Moreover, in these cases, the resulting paper space is homeomorphic to a sphere.

We are also interested in a particular class of paper folding schemes, known as plain paper folding schemes. These constitute both the most fundamental and the most frequently occurring instances within the theory. Furthermore, in such cases, the associated paper space is homeomorphic to a sphere. For this class, in addition to the paper surfaces being Ahlfors 2-regular and linearly locally contractible, we shall see that they are quasisymmetrically equivalent to the $2$‑sphere. In particular, the uniformization result holds precisely in the setting of plain paper folding schemes, where the resulting paper space is homeomorphic to a sphere.

Specifically, let $ \gamma $ be a polygonal arc or a simple polygonal closed curve. Two pairs of (not necessarily distinct) points $ \{x, x'\} $ and $ \{y, y'\} $ on $ \gamma $ are \textit{unlinked} if one of the pairs is contained in the closure of a connected component of the complement of the other; otherwise, they are considered \textit{linked}. A symmetric and reflexive relation $ R $ on $ \gamma $ is \textit{unlinked} if any two unrelated pairs of related points are unlinked, that is, if $ x R x' $, $ y R y' $, and neither $ x $ nor $ x' $ is related to either $ y $ or $ y' $, then $ \{x, x'\} $ and $ \{y, y'\} $ are unlinked. An interior-disjoint collection $ \mathcal{P} $ of segment pairings on $ \gamma $ is unlinked if the relation $ \mathcal{P} $ is unlinked. Finally, a \textit{plain paper-folding scheme} $ (P, \mathcal{P}) $ is defined as a scheme where $ P $ consists of a single polygon, and the relation $ \mathcal{P} $ is unlinked.

\begin{example}\label{ex13}
\normalfont
Consider $\{a_i\}_{i\in \mathbb{N}}$ a sequence of positive real numbers such that $\sum a_i=\frac{1}{2}$. Let $P=\{a+ib \in \mathbb{C}: \ 0\le a,b \le 1\}$ be the unit square, and let $\mathcal{P}$ be a pairing that glues the horizontal sides of $P$ and folds the right side of $P$ in half. On the left side, we define a countable  number of folds of length $a_i$, denoted  $\{\langle \alpha_i,\alpha_i'\rangle\}_{i\in \mathbb{N} }$. See Figure~\ref{fig2}.

 Here, $\alpha_0$ is the segment with length $a_0$ starting at the upper left vertex of $P$, and $\alpha_0'$ is the segment of $\partial P$ with length $a_0$  starting at the endpoint of $\alpha_0 $. Similarly, $\alpha_1$ is the segment of $\partial P$ with length $a_1$  starting  at the endpoint of $\alpha_0'$, and so on.  This process continues such that the left side of $P$ is covered, except for its lower vertex, by folds arranged from top to bottom.
\end{example}

\begin{figure}[htb]
	\centering
	\includegraphics[width=0.36\textwidth]{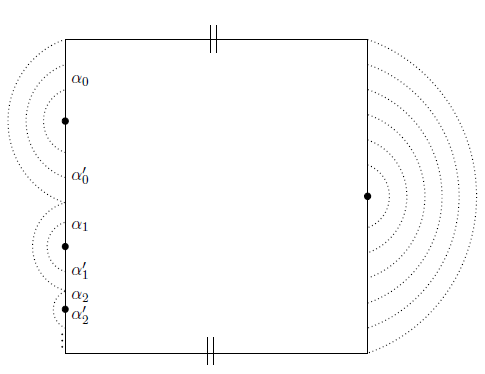}
	\caption{Paper Space of Example \ref{ex13}.}
    \label{fig2}
	\end{figure}
    Building on this, the points of a paper space are classified according to the following criteria.
    For $k \in \mathbb{N} \cup \{ \infty \}$, a point $x \in G=\pi(\partial P)$ is a \textit{vertex of valence $k$} or simply a $k$\textit{-vertex} if $\#\pi^{-1}(x) = k \neq 2,$ or $\#\pi^{-1}(x) = k = 2$ and $\pi^{-1}(x)$ contains a vertex of $P$ (here $\#A$ denotes the cardinality of the set $A$). Points of $G$ that are not vertices or accumulations of vertices are \textit{planar points}.
 A point is a \textit{regular vertex} if it is not an accumulation of vertices or an $\infty$- vertex. 
  A point is \textit{singular} if it is an $\infty$-vertex or is an accumulation of vertices. 

Now, observe that if $ \langle \alpha, \alpha' \rangle \in \mathcal{P} $ and $ x \in \operatorname{int}(\alpha) $, then $ \pi_P(x) \in S_P $ is a planar point, because the preimage $ \pi_P^{-1}(\pi_P(x)) $ consists of exactly two points, neither of which is a vertex of $ P $. This distinction explains why the definition above differentiates between 2-vertices and $ k $-vertices when $ k \neq 2 $.

On the other hand, the endpoints of paired segments are not necessarily projected onto conic or non-regular points. Given $ \langle \alpha, \alpha' \rangle \in P $, both segments can be decomposed into two subsegments $ \alpha_1 \cup \alpha_2 $ and $ \alpha'_1 \cup \alpha'_2 $ sharing a common point. In this case, $ \alpha_1 $ is paired with $ \alpha'_1 $, and $ \alpha_2 $ with $ \alpha'_2 $, forming a new pairing $ P' $. Since $ S_P = S_{P'} $, the common endpoint of $ \alpha_1 $ and $ \alpha_2 $ is projected onto a planar point in the paper space.

However, it is always possible to choose a paper folding $ (P, \mathcal{P}) $ such that the endpoints of any segment pairing are projected onto non-planar points in the paper space. These configurations are known as \textit{efficient paper foldings} (see \cite{Marceldisserta}, Section 2.4). %In this thesis, it will always be assumed, without loss of generality, that any paper folding $ (P, P) $ is efficient.

\begin{theorem}
     (Topological structure of a plain paper folding \cite{de2012paper}). The quotient $S$ of a plain paper folding scheme is a topological sphere.
\end{theorem}

%According to Theorem \ref{teo1-29}, the metric $ d_S $ of a paper space is intrinsic. Since $ P $ is compact and the projection $ \pi $ is continuous and surjective, $ S = \pi(P) $ is also compact. Thus, by Theorem \ref{teo1-1-11}, $ d_S $ is strictly intrinsic. In other words, any two points in a paper space are either infinitely far apart or can be joined by a path whose $ d_S $-length matches the distance between them. 

%Furthermore, due to the compactness of $ S $, Theorem \ref{teo1-28} ensures that $ (S, d_S) $ is topologically homeomorphic to the quotient of $ P $ induced by the decomposition $ \partial P $. As $ \pi $ is a closed map and $ \partial P $ is a closed subset of $ P $, the quotient $ G = \pi(\partial P) $ forms a closed subset within $ S $.

In conclusion, as guaranteed by the discussions in previous sections, the metric $ d_S $ of a paper space is intrinsic. Given that $ P $ is compact and the projection map $ \pi $ is continuous and surjective, it follows that $ S = \pi(P) $ is also compact. Consequently, $ d_S $ is strictly intrinsic, meaning that any two points in a paper space are either infinitely far apart or can be joined by a path whose $ d_S $-length exactly matches the distance between them.

Furthermore, due to the compactness of $ S $, it follows that $ (S, d_S) $ is topologically homeomorphic to the quotient of $ P $ induced by the pairings on $ \partial P $. Since $ \pi $ is a closed map and $ \partial P $ is a closed subset of $ P $, the quotient $ G = \pi(\partial P) $ forms a closed subset within $ S $.

\subsection{Metric structure of paper spaces}\label{sec:metric_structure}

This section presents an analysis of the metric properties of paper spaces, with more detailed information available in \cite{de2012paper} and \cite{BBI}. Let $ S $ be a paper space. It will be seen that in a neighborhood of planar points, $ S $ is locally Euclidean, meaning that such points have an open neighborhood isometric to an open ball in $ \mathbb{R}^2 $. Similarly, in a neighborhood of a regular vertex $ x $, the space $ S $ is isometric to the apex of a cone, where the cone angle corresponds to the sum of the internal angles of the multipolygon $ P $ at the points of $ \pi^{-1}(x) $.

Throughout this section, $ (P, \mathcal{P}) $ represents an arbitrary paper-folding scheme and its associated paper space $ S $.

For a non-singular point $ x \in G $, the preimage under $ \pi $ is given by
$$
\pi^{-1}(x) = \{\tilde{x}_1, \dots, \tilde{x}_k\}.
$$
The internal angle of $ P $ at each $ \tilde{x}_i $ is denoted by $ \tilde{\eta}_i $, and the \textit{cone angle} at $ x $ is defined as
$$
\eta(x) = \tilde{\eta}_1 + \cdots + \tilde{\eta}_k.
$$
In particular, if $ x \in G $ is a planar point, then its cone angle is
$$
\eta(x) = 2\pi.
$$

Now, consider the topological concept of a \textit{cone}. The cone $ \operatorname{Cone}(X) $ over a topological space $ X $ is the quotient of $ [0,\infty) \times X $ by the equivalence relation that collapses $ \{0\} \times X $ to a point, which is called the \textit{origin} or \textit{apex} of the cone. If $ (X,d) $ is a metric space, then the cone can be given a metric, where for $ p = [t, x] $ and $ q = [s, y] \in \operatorname{Cone}(X) $, the distance is defined by:
$$
d_c(p,q)= \left\{
\begin{matrix}
\sqrt{t^2+s^2-2st\cos (d(x,y))}, & \text{if } d(x,y) \le \pi, \\
t+s & \text{if } d(x,y) \ge \pi.
\end{matrix}
\right.
$$

An important result is that if $ d $ is strictly intrinsic, then $ d_c $ is also strictly intrinsic, as shown in \cite{BBI}, Theorem 3.6.17.

Another important fact is that each planar point has an isometric neighborhood to a flat Euclidean ball, while each regular vertex has an isometric neighborhood to a ball centered at the vertex of a metric cone on a circle. Specifically, let $ S^1_r $ represent the circle of radius $ r > 0 $ in $ \mathbb{R}^2 $ equipped with the intrinsic metric. The corresponding cone, $ \operatorname{Cone}(S^1_r) $, is locally isometric to the Euclidean plane everywhere except at the apex when $ r \neq 1 $. In the special case where $ r = 1 $, the cone $ \operatorname{Cone}(S^1_1) $ is globally isometric to $ \mathbb{R}^2 $. These results can be found in \cite{de2012paper}.  %These results were stated in \cite{de2012paper}.

\begin{definition}
    A \textit{conic-flat surface} $S$ is a metric space that is locally isometric to cones on circles. Specifically, for every $p\in S$, there exist $r,\epsilon>0$, $x\in \operatorname{Cone}(S_r^1),$ and an isometry from $B_S(p,\epsilon)$ onto $B_{\operatorname{Cone}(S_r^1)}(x,\epsilon),$ where $S_r^1$ is a circle in $\mathbb{R}^2$ of radius $r$ with the intrinsic metric. This implies that there are two types of points on a conic-flat surface $S$: those at which $S$ is locally Euclidean and those where $S$ is locally isometric to a neighborhood of the apex of $\operatorname{Cone}(S_r^1)$. In the latter case, $2\pi r$ is the \textit{cone angle} at the point.
\end{definition}

The \textit{metric structure} of paper spaces, as discussed in \cite{de2012paper}, guarantees that for a non-singular point $ x $ of $ G $, there exists a radius $ r > 0 $ such that the ball $ B_S(x, r) $ is isometric to the ball $ B_{\operatorname{Cone}(S^1_{\eta(x)/2\pi})}(0, r) $. In particular, the complement $ S \setminus \mathcal{V}^s $, where $ \mathcal{V}^s $ denotes the set of singular points, forms a conic-flat surface. Furthermore, it is established that the metric $ d_S $ on a paper space $ S $ is strictly intrinsic.

Observe that if $ S $ has no singular points, this result ensures that every point in $ S $ has a neighborhood isometric to a ball centered at the apex of a cone over a circle.

\subsection{Quasisymmetric uniformization}\label{sec:qs_unif}

This section explores a metric space $ S $ that is homeomorphic to $ S^2 $ and examines the conditions under which $ S $ can be mapped onto $ S^2 $ through a quasisymmetric homomorphism. The main references used in this section are \cite{bonk2002quasisymmetric}, \cite{ahlfors1954}, \cite{falconer2004fractal}, and \cite{de2005extensions}.

\subsubsection{Quasisymmetric maps}Working directly with local distortion properties of mappings can be challenging. For this reason, it is often convenient to impose a stronger, global condition on the map $f$ to facilitate analysis and application. Quasisymmetric maps provide such a framework in the context of metric spaces, extending the notion of controlled distortion beyond the infinitesimal scale. %Due to the infinitesimal nature of quasiconformality, it can be challenging to work with directly. Therefore, it is often necessary to impose a stronger, global condition on the mapping $f$ to facilitate analysis and application. Quasisymmetric maps are a generalization of (quasi)conformal maps to metric spaces. Specifically, a homeomorphism between compact Riemannian manifolds is quasisymmetric if and only if it is quasiconformal.

\begin{definition}\label{def-qs}
	
	A homeomorphism $f:X\to Y$ is called $\eta$-quasisymmetric if there exists a homeomorphism $\eta: [0,\infty)\rightarrow [0,\infty)$ such that
	\vspace{-0.08cm}
	\begin{equation}\label{qs-eq2}
		d_X(x,a)\le td_X(x,b) \Rightarrow d_Y(f(x),f(a)) \le \eta(t) d_Y(f(x),f(b))
	\end{equation}
	for every $a,b,x \in X$ and every $t>0$. 
\end{definition}

\begin{figure}[!h]
	\centering
	\includegraphics[scale=0.60]{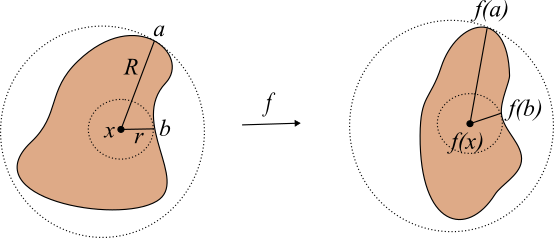}
	\caption{Quasisymmetric map. }
\end{figure}

The geometric meaning of this condition is that balls are mapped to “round” sets with quantitative control for their “eccentricity”.
This means that the image of a ball under a quasisymmetric map will not be excessively stretched or distorted, maintaining a controlled shape. $f$ is quasisymmetric if it distorts relative distances by a limited amount that depends on $t$.

From the definition, it is not difficult to prove that quasisymmetric homeomorphisms between reasonable spaces possess many strong properties: they are Hölder continuous, and quasisymmetry extends to limit homeomorphisms. Furthermore, note that quasisymmetric maps are always injective and continuous. The composition of two quasisymmetric maps is again quasisymmetric, and the inverse of a surjective quasisymmetric map is also quasisymmetric. Two metric spaces are said to be \textit{quasisymmetrically equivalent} if there is a quasisymmetric map from one onto the other \cite{davidsemmes}.

The map $f(x)=\lambda x$ on $\mathbb{R}^n$ is an example of quasisymmetric map with $\eta(t)=t$, for each $\lambda>0.$

\subsubsection{Hausdorff measure}%Another important concept here is the Hausdorff measure. 
The Hausdorff measure is a generalization of the classical notions of area and volume to non-integer dimensions. This measure is indispensable for understanding the size of the smallest subsets of a metric space, especially in fractal geometry. The main reference for this section is \cite{falconer2004fractal}.

Let $ \{C_i\} $ be a countable (or finite) collection of sets with diameters at most $ \delta $ that cover $ C $, i.e., $ C \subset \bigcup_{i=1}^\infty C_i $ with $ 0 < \operatorname{diam}(C_i) \leq \delta $ for each $ i $. In this case, we say that $ \{C_i\} $ is a $ \delta $-\textit{cover} of $ C $.

Now, for $ 0 \leq s < \infty $ and $ A \subset X $, the Hausdorff measure of $ A $ is defined as
$$
\mathcal{H}_\delta^s(A) = \inf \left\{ \sum_j \operatorname{diam}(C_j)^s \; ; \; \{C_i\} \text{ is a } \delta\text{-cover of } A \right\},
$$
where the infimum is taken over all countable covers of $ A $ by sets $ C_j \subset X $ satisfying $ \operatorname{diam}(C_j) < \delta $. Here, the diameter of a set $ C $ is given by
$$
\operatorname{diam}(C) = \sup \left\{ d_X(x, y) \; ; \; x, y \in C \right\},
$$
which is the greatest distance between any pair of points in $ C $.

As $ \delta $ decreases, the class of permissible covers of $ A $ is reduced. Therefore, the infimum $ \mathcal{H}_\delta^s(A) $ increases, or at least does not decrease, as $ \delta \to 0 $, and thus approaches a limit. The \textit{$s$-dimensional Hausdorff measure} of $ A $ is then defined as
$$
\mathcal{H}^s(A) = \lim_{\delta \rightarrow 0} \mathcal{H}_\delta^s(A) = \sup_{\delta > 0} \mathcal{H}_\delta^s(A).
$$

Hausdorff measures generalize the familiar ideas of length, area, volume, and so on. For subsets of $ \mathbb{R}^n $, the $ n $-dimensional Hausdorff measure is, up to a constant multiple, just the $ n $-dimensional Lebesgue measure, which corresponds to the usual $ n $-dimensional volume. More precisely, if $ C $ is a Borel subset of $ \mathbb{R}^n $, then
$$
\mathcal{H}^n(C) = c_n^{-1} \operatorname{vol}^n(C),
$$
where $ c_n $ is the volume of a unit $ n $-dimensional ball.

\begin{lemma}\label{equivhausdorff}
     Let $X$ be a metric space and $d$ and $d'$ two Lipschitz equivalent metrics in $X$.  The $s$- dimensional Hausdorff measure related to each metric is also equivalent. 
\end{lemma}

\begin{proof}
    Since $d$ and $d'$ are equivalent metrics, that is, there exists a constant $L$ such that $\dfrac{1}{L}d'\le d\le Ld'$. Therefore, for any subset $E\subset X$, we have:
    $$\dfrac{1}{L^s}\mathcal{H}^s_{d,\delta}(E)\le \mathcal{H}^s_{d',\delta}(E)\le L^s \mathcal{H}^s_{d,\delta}(E).$$
\end{proof}

Notice that, the Hausdorff measure is invariant under isometries of the metric space. This means that the measure does not change if the set is translated, rotated, or reflected.

\subsubsection{Definitions of regularity}
Among metric spaces, Ahlfors regular spaces have been extensively studied over the past few decades. These spaces have proven to be an excellent setting for the development of harmonic analysis \cite{david2000regular}. One of the main tools used in this work is the condition of being linearly locally connected and Ahlfors regular.

\begin{definition}

A metric space ${X}$ is  \textit{Ahlfors} $Q$\textit{-regular} if there is a constant $C > 0$ such that the $Q$-dimensional Hausdorff measure $\mathcal{H}^Q$ of every closed  $r$-ball  $\overline{B}(a,r)$ satisfies
$$C^{-1}r^Q \leq \mathcal{H}^Q(\overline{B}(a,r)) \leq Cr^Q$$
for all  $0 < r \leq \operatorname{diam}({X})$.
    
\end{definition}

This condition ensures that the measure of balls in the space grows polynomially with the radius, which implies a uniform distribution of the measure across the space.  

\begin{definition}
    A metric space $(Z,d)$ is called $\lambda$-{\textit{linearly locally contractible}} where $\lambda \ge 1$, if every ball $B(a,r)$ in $Z$ with $0 < r \leq \operatorname{diam}{Z}/ {\lambda}$ is contractible inside  $B(a,\lambda r)$. This means there exists a continuous map $H : B(a,r) \times [0,1] \rightarrow B(a,\lambda r)$ such that $H(\cdot, 0)$ is the identity on $B(a,r)$ and $H(\cdot, 1)$ is a constant map. 

    The space is called \textit{linearly locally contractile} if it is $\lambda$-linearly locally contractile for some $\lambda \ge 1$.

\end{definition}

This property ensures that each small ball in the space can be continuously shrunk to a point within a slightly larger ball.

These two conditions often appear together in the study of metric spaces, particularly in the context of geometric group theory and analysis on metric spaces. For example, if a space is Ahlfors $n$-regular and linearly locally contractible, it often satisfies important geometric and analytic properties, such as Poincaré inequalities.

\begin{definition}
  A metric space $(Z,d)$ is called $\lambda\textit{-LLC}$ for $\lambda \ge 1$ (LLC stands for \textit{linearly locally connected}) if the following two conditions are satisfied:
\begin{itemize}
     \item $(\lambda$-LLC$_1)$ If $B(a,r)$ is a ball in $Z$ and $x,y \in B(a,r)$, then there is a continuum $E \subset B( a,\lambda r)$ containing $x$ and $y$.\\
    
     \begin{figure}[!h]
          \centering
        \includegraphics[scale=0.4]{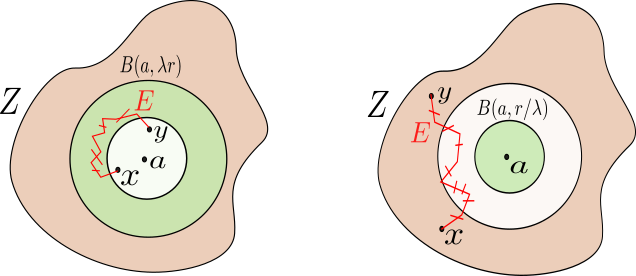}
            \caption{LLC$_1$ and LLC$_2$ conditions, respectively.}
            \label{fig:my_label}
        \end{figure}
     \item $(\lambda$-LLC$_2)$ If $B(a,r)$ is a ball in $Z$ and $x,y \in Z \setminus B(a,r)$, then there is a continuum $E \subset Z \setminus B(a, {r}/{\lambda})$ containing $x$ and $y$.
\end{itemize}
\end{definition}

Recall that a \textit{continuum} is a compact and connected set that has more than one point.

Roughly speaking, the LLC condition excludes cusps from the geometry of a metric space. The term \textit{linearly locally connected} (LLC) originates from the following property: Let $(X, d)$ be a metric space satisfying the $\lambda$-LLC$_1$ condition, and consider a point $x \in X$ and a radius $r > 0$. If $C(x)$ denotes the connected component of the ball $B(x, r)$ that contains $x$, then the following inclusion holds:

$$
B(x, r/\lambda) \subseteq C(x) \subseteq B(x, r).
$$

The following lemma states that linearly local contractibility implies the LLC (Linearly Locally Connected) condition for compact connected topological $ n $-manifolds, and is equivalent to it when $ n = 2 $.

 \begin{lemma}[\cite{bonk2002quasisymmetric}]
Suppose  $Z$  is a metric space which is a compact connected topological $n$-manifold. Then:
\begin{enumerate}
     \item If $Z$ is $\lambda$-linearly locally contractile, then $Z$ is $\lambda'$-LLC for each $\lambda'>\lambda$.
     \item If $n=2$ and $Z$ is LLC, then $Z$ is linearly locally contractible. The linear local contractibility constant depends on $Z$ and not just on the LLC constant.
\end{enumerate}
\end{lemma}

The following theorem provides a uniformization result for metric spheres. This is analogous to the classical uniformization theorem for Riemann surfaces, which states that every simply connected Riemann surface is conformally equivalent to the unit disk, the complex plane, or the Riemann sphere.

\begin{theorem}[\cite{bonk2002quasisymmetric}]\label{teo1}
Let $Z$ be an Ahlfors 2-regular metric space homeomorphic to $S^2$. Then $Z$ is quasisymmetric to $S^2$ if $Z$ is linearly locally contractible.
\end{theorem}

\section{Quasisymmetric Uniformization of paper spaces}\label{sec-surfL}

%This section introduces a special class of \emph{plain paper spaces}, which are the main objects of study in this paper. A plain paper space is constructed from a pair $ (P, \mathcal{P}) $, where $ P $ is a finite union of planar polygons in the complex plane, and $ \mathcal{P} $ is a collection of unlinked pairings on the boundary $ \partial P $. These pairings are disjoint in their interiors and follow specific structural patterns.

This section introduces a special class of \emph{paper spaces}, which constitute the central objects of study in this work. Recall that a paper space is constructed from a pair $(P, \mathcal{P})$, where $P$ is a multipolygon and $\mathcal{P}$ is a collection of pairings on the boundary $\partial P$. These pairings follow specific structural patterns, whose precise description will be given within this section.

%Two types of pairings are considered: segment pairings, which identify segments of equal length either along the same side or across different sides of the multipolygon; and Type~$ W $ identifications, which involve an infinite sequence of decreasing segments that are alternately paired along an edge.

We consider two types of pairings: segment pairings, previously introduced as identifications of segments of equal length either along the same side or across different sides of the multipolygon; and Type~$W$ identifications, consisting of an infinite sequence of decreasing segments alternately paired along an edge.

Figures~\ref{figsupL} and \ref{4ret} illustrate examples of paper spaces constructed using these pairings. In one case, a polygon features a Type~$ W $ identification on one side and segment pairings on the others. In another, four rectangles are joined to form a multipolygon, with both finite and infinite identifications along their edges.

To establish notation, we denote the resulting plain paper space as the metric quotient $ \mathcal{L}_{\mathcal{P}} := P / d_P^{\mathcal{P}} $. The associated projection is written as $ \pi_{\mathcal{P}} \colon P \to \mathcal{L}_{\mathcal{P}} $ whose restriction to the interior of $P$ is a homeomorphism onto the complement of the glued boundary $ G_{\mathcal{P}} := \pi_{\mathcal{P}}(\partial P) $. This notation will be used throughout the remainder of the paper.

This construction gives rise to a class of surfaces denoted by $ \mathcal{L} $, consisting of all spaces $ \mathcal{L}_{\mathcal{P}} $ arising as metric quotients from such pairings. In addition, we consider a distinguished subclass of $\mathcal{L}$, denoted by $\mathcal{L}^*$, formed by the plain paper spaces contained in $\mathcal{L}$. Note that a plain paper folding scheme is defined from a single polygon. However, situations analogous to Figure \ref{fig:papel3} may occur, in which a union of multipolygons effectively forms a single polygon. In this case, when viewed as one polygon, the pairings $\mathcal{P}$ are unlinked, so the decomposition into multiple polygons can be regarded as a single polygon with unlinked pairings.

\begin{figure}[h!]
    \centering
    \includegraphics[width=0.5\linewidth]{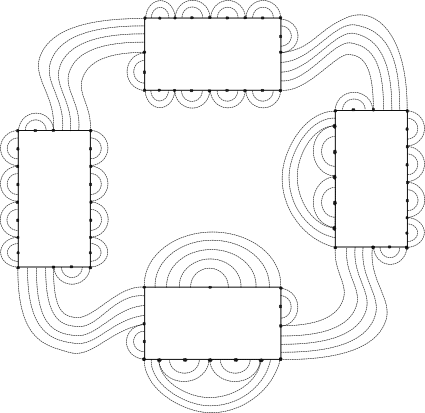}
    \includegraphics[width=0.5\linewidth]{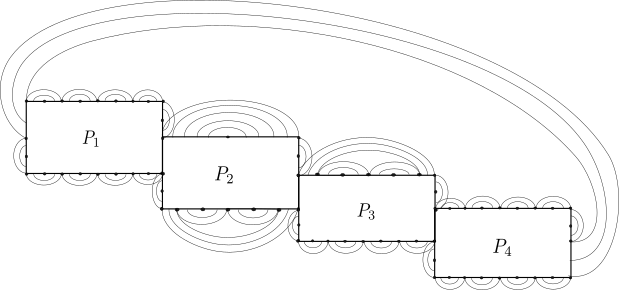}
    \caption{Paper space obtained as the quotient of four rectangles.}
    \label{fig:papel3}
\end{figure}

The remainder of this section is dedicated to the proof of our main theorem, which we restate below. In Section 3.1, we introduce the class of surfaces under consideration. Section 3.2 analyzes the geometry of metric balls in these new spaces. Finally, in Section 3.3, we verify that these spaces satisfy the conditions of being Ahlfors 2-regular and linearly locally contractible.

\begin{theorem}\label{teo_principal}
The surfaces in the class $ \mathcal{L} $ satisfy the following properties:
\begin{itemize}
    \item[(a)] Ahlfors 2-regularity;
    \item[(b)] Linear local contractibility.
\end{itemize}
As a consequence, the surfaces of type $ \mathcal{L}^* $ are quasisymmetrically equivalent to the 2-sphere.
\end{theorem}

%This section introduces a special class of \emph{paper spaces}, which constitute the main objects of study throughout this paper. The construction is exemplified by the paper space shown in Figures~\ref{figsupL} and \ref{4ret}.

\subsection{The construction}Let $(P, \mathcal{P})$ be a paper folding scheme, whose pairings follow specific structural patterns as described below.

\begin{figure}[!h]
    \centering
    \includegraphics[scale=8]{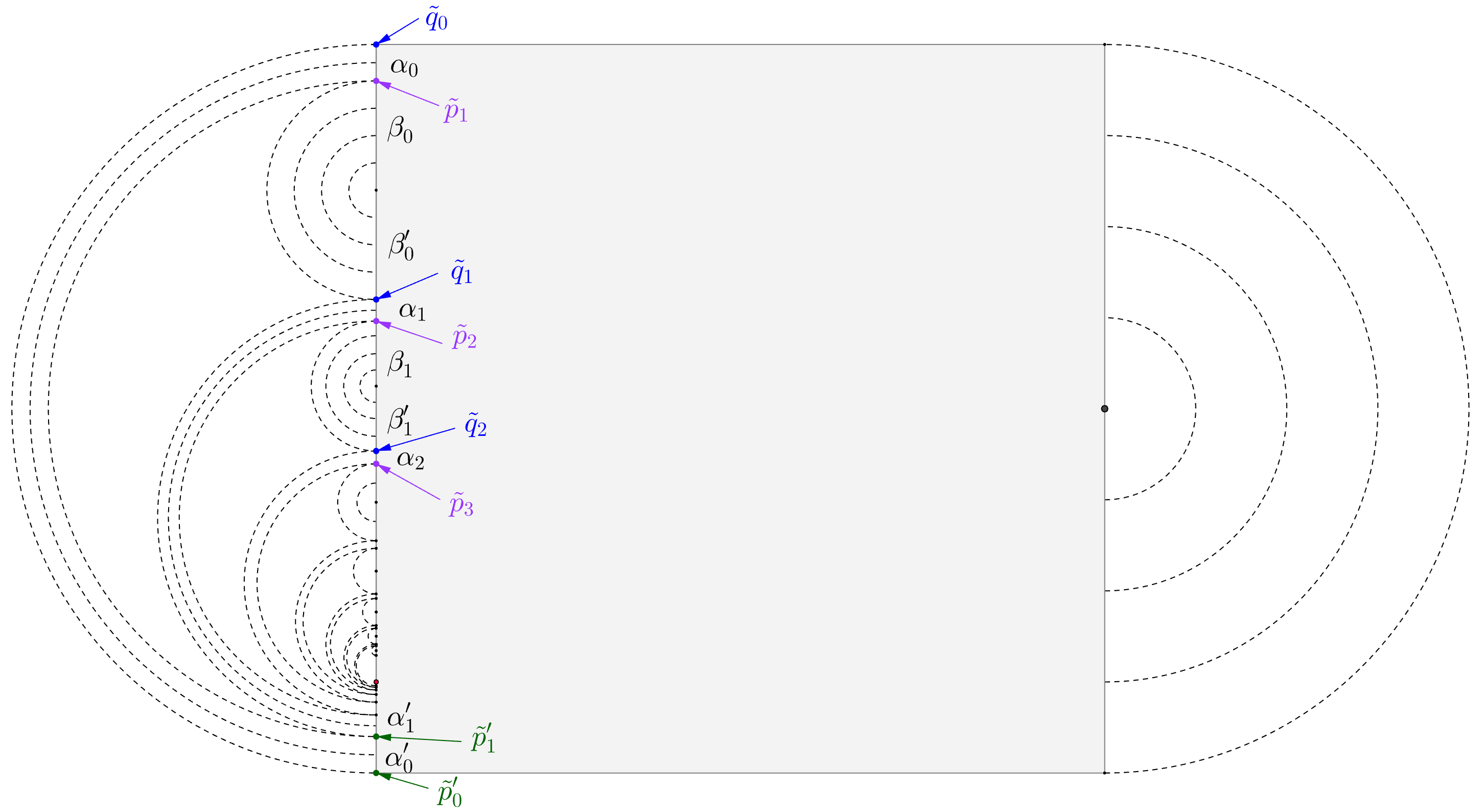}
    \caption{An example of a paper space, belonging to the class of surfaces denoted by $ \mathcal{L}^* $.}
    \label{figsupL}
\end{figure}

The pairings considered fall into two categories: \emph{basic segment pairings} (Figure~\ref{fig:seg_pairing}), which identify segments of equal length either along the same side or across different sides of the multipolygon, and \emph{Type~$W$ identifications}, defined below. Throughout this paper, we assume that the collection $ \mathcal{P} $ consists of a finite number of basic segment pairings and a finite collection of type $W$ identifications, which is a specific type of an infinite collection of pairings.
%Whenever an infinite number of pairings appears along a side, they are required to be of Type~$W$.

Without loss of generality, we assume in the following discussion that the pairing types below are positioned along a vertical edge of the multipolygon. An analogous reasoning applies to horizontal edges. In such cases, references to the ``upper'' and ``lower'' vertices should be interpreted as referring to the ``left'' and ``right'' vertices, respectively.

\bigskip

\paragraph{\textit{Infinite Type~$W$ Identifications.}}% Let $ \lambda > 1 $ be a real number. 
Consider two real sequences $ \{a_i\}_{i \in \mathbb{N}} $ and $ \{b_i\}_{i \in \mathbb{N}} $ satisfying
$$
\sum_{i=0}^{\infty} a_i + \sum_{i=0}^{\infty} b_i = \frac{\ell}{2}, 
$$where $\ell$ denotes a constant.
%and assume they decrease geometrically:
%$$
%a_k = a_0 \lambda^{-k}, \quad b_k = b_0 \lambda^{-k}.
%$$
The infinite pairing pattern along one side of the polygon $ P $ of length $\ell$ is described as follows (Figure~\ref{figsupL}):
\begin{enumerate}[(i)]
    \item $\alpha_0$: the segment of length $ a_0 $ starting at the upper-left vertex $ \tilde{q_0} $ is paired with $ \alpha'_0 $, a segment of the same length ending at the lower-left vertex $ \tilde{p}_0' $;
    \item $\beta_0$: the segment of length $ b_0 $ starting at the endpoint of $ \alpha_0 $, denoted $ \tilde{p}_1 $, is paired with $ \beta'_0 $, starting at the endpoint of $ \beta_0 $;
    \item $\alpha_1$: the segment of length $ a_1 $ starting at the endpoint of $ \beta'_0 $, denoted $ \tilde{q_1} $, is paired with $ \alpha'_1 $, starting at the endpoint of $ \alpha'_0 $, denoted $ {\tilde{p}_1}' $;
    \item And so on, so that the left side of $ P $ is covered%---except for a vertical gap of height $ \sum a_i $---
    from top to bottom by the sequence $ \alpha_0, \beta_0, \beta'_0, \alpha_1, \beta_1, \beta'_1, \dots $, and from bottom to top by $ \alpha'_0, \alpha'_1, \alpha'_2, \dots $.
\end{enumerate}

Each segment $ \alpha_i $ is glued to its corresponding segment $ \alpha'_i $, and each segment $ \beta_i $ to its corresponding segment $ \beta'_i $.

\bigskip

\paragraph{\textit{Finite Type~$W$ Identifications.}} This is a finite analogue of the Type~$W$ construction. Here, finite sequences $ \{a_i\}_{i=0}^{n} $ and $ \{b_i\}_{i=0}^{n} $ are considered, satisfying
$$
\sum_{i=0}^{n} a_i + \sum_{i=0}^{n} b_i = \frac{\ell}{2},
$$
 with pairings constructed similarly but terminating after a finite number of segments, where $\ell$ denotes a constant that represents a side of the polygon.

\bigskip

We note that Finite Type~$W$ identifications, although presented separately, can be interpreted as finite unions of basic segment pairings. Likewise, Infinite Type~$W$ identifications are naturally understood as infinite collections of such pairings.

%Figure~\ref{figsupL} shows a paper space constructed using the types of pairings described above. In this example, the left side exhibits Type~$W$ identifications; the right side is folded in half via a basic segment pairing; and the top and bottom sides are glued entirely to one another via a basic segment pairing.

Figures~\ref{figsupL} and \ref{4ret} illustrate paper spaces constructed using the types of pairings described above. In Figure~\ref{figsupL}, the left side exhibits Type~$W$ identifications; the right side is folded in half by a basic segment pairing; and the top and bottom sides are glued to each other through another basic segment pairing. In Figure~\ref{4ret}, three rectangles and a pentagon are combined to form a multipolygon by identifying four of their sides using basic segment pairings. Along the sides of the rectangles, we observe, for instance, finite basic segment pairings that fold segments in half, as well as infinite Type~$W$ identifications on two of the sides. Other sides may also be folded in half, depending on the configuration.

%Now, consider the paper space associated with the paper folding $ (P, \mathcal{P}) $, defined as the metric quotient $ \mathcal{L}_\mathcal{P}:= P / d_P^{\mathcal{P}} $. Since $ (P,d) $ is a compact metric space, it follows that the metric quotient is homeomorphic to the topological quotient $ P / {\sim_{\mathcal{P}}} $. Therefore, the paper space $ (\mathcal{L}_\mathcal{P}, d^{\mathcal{P}}) $ is itself a compact metric space.

\begin{figure}[!h]
    \centering
    \includegraphics[scale=0.8]{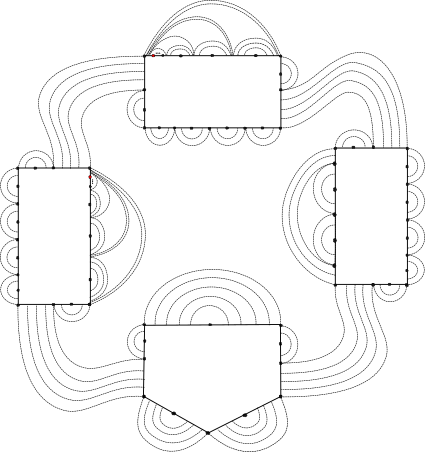}
    \caption{A paper space in the class $\mathcal{L}$ but not in $\mathcal{L}^*$, constructed as the quotient of four polygons.}
    \label{4ret}
\end{figure}

%Furthermore, let $ \pi_{\mathcal{P}} : P \rightarrow \mathcal{L}_\mathcal{P} $ denote the projection induced by the relation $ \sim_{\mathcal{P}} $. If we define
%$$
%G_{\mathcal{P}} := \pi_{\mathcal{P}}(\partial P),
%$$
%then the restriction $ \pi_{\mathcal{P}} : \operatorname{int}(P) \rightarrow \mathcal{L}_\mathcal{P} \setminus G_{\mathcal{P}} $ is a homeomorphism.

Since 
$$
\mathcal{L}_{\mathcal{P}} := P / d_P^{\mathcal{P}}
$$
denotes the metric quotient of a planar multipolygon $P$ under a collection of admissible pairings $\mathcal{P}$ on its boundary, we define the following classes of surfaces. The class $\mathcal{L}$ consists of all such metric spaces arising from paper folding schemes:
$$
\mathcal{L} := \left\{ \mathcal{L}_{\mathcal{P}} \,\middle|\, (P, \mathcal{P}) \ \text{is a paper folding scheme as described above} \right\},
$$
and the subclass $\mathcal{L}^*$ is given by those spaces obtained from plain paper folding schemes:
$$
\mathcal{L}^* := \left\{ \mathcal{L}_{\mathcal{P}} \in \mathcal{L} \,\middle|\, (P, \mathcal{P}) \ \text{is a plain paper folding scheme} \right\}.
$$

These two classes will serve as the main objects of study in the subsequent sections.

\subsection{Preimages of the balls under the projection map} \label{bolas-sem-ling}

Before proving the main theorem, it is necessary to develop a deeper understanding of the geometry of the paper spaces under consideration. More precisely, we need to analyze the shape of metric balls in these spaces. Since the identifications are made only along the boundary of the multipolygon, it is sufficient to study the \emph{preimages} of metric balls under the projection $ \pi_{\mathcal{P}} \colon P \to \mathcal{L}_\mathcal{P} $, where $ \mathcal{L}_\mathcal{P} \in \mathcal{L} $ is a paper surface (see {Theorem \ref{univ.thm}}).

This section is devoted to the study of such preimages, with a focus on three distinct cases based on the location of the center of the ball: regular interior points, accumulation (or singular)  points along the boundary, and conic points. 
%This section examines the preimages of metric balls on the surface $ \mathcal{L} $ under the projection $ \pi_\mathcal{P}: P \to \mathcal{L} $. The analysis focuses on three cases: balls centered at regular points, accumulation points, and conic points.

Recall that the identifications along the boundary of the multipolygon $ P $ can take one of the following forms:  a finite collection of basic segment pairings;  a finite collection of infinite Type~$ W $ identifications; or a finite combination of both basic segment pairings and finite collection of infinite Type~$ W $ identifications.

%To determine the shape of these metric balls, it suffices to analyze their preimages under the projection $\pi_\mathcal{P}$, since the metric quotient respects the identifications along the boundary (see {Theorem \ref{univ.thm}}). 

%To avoid ambiguity, we clarify that throughout this section, any reference to a metric ball implicitly refers to its \emph{preimage} under the projection map $ \pi_{\mathcal{P}} \colon P \to \mathcal{L}_{\mathcal{P}} $. In particular, when we refer to a \emph{conic point of angle $ k\pi $} (with $ k \in \mathbb{N} $), we are referring to one of its $ k $ preimages in the boundary domain $ P $.

To avoid ambiguity, we clarify that throughout this section, any reference to a metric ball, or simply a ball, implicitly refers to its \emph{preimage} under the projection map 
$$
\pi_{\mathcal{P}} \colon P \to \mathcal{L}_{\mathcal{P}}.
$$
In particular, when we refer to a conic point or \emph{conic point of angle $k\pi$} (with $k \in \mathbb{N}$), we are referring to one of its $k$ preimages in the boundary domain $P$. That is, if 
$x \in \mathcal{L}_{\mathcal{P}}$ is a conic point of angle $k\pi$, we denote
$$
\pi_{\mathcal{P}}^{-1}(x) = \{\tilde{x}^1, \ldots, \tilde{x}^k\} \subset P,
$$
where each $\tilde{x}^i \in P$ is also referred to as a conic point and represents a preimage of $x$ in the boundary domain. In this sense, whenever we write $\tilde{x}$, we are implicitly choosing a representative in $\pi_{\mathcal{P}}^{-1}(x)$.

We call a \emph{semi-ball}, defined as the intersection of a ball centered at a point on the boundary $ \partial P $ with the interior of the multipolygon $ P $. That is, it represents the portion of the ball that lies entirely within $ P $. We denote such a set by $ {\SB}(\tilde{p}, r) $, the \emph{semi-ball} of radius $ r $ centered at $ \tilde{{p}} \in \partial P $.

Moreover, from this point forward, we will work with the \emph{maximum metric} (also known as the $ \ell_\infty $ metric) on $ P $ instead of the Euclidean metric. This choice is made purely for convenience, as it simplifies the geometric analysis without affecting the generality of the arguments.

Before addressing the more general cases, we note the following examples. 
If the radius $r$ is sufficiently small and $\tilde{p}$ is a conic point with angle $\pi$ such that 
$B(\tilde{p},r)$ does not contain other conic points with different angles, then $B(\tilde{p},r)$ is a 
\emph{semi-ball} centered at $\tilde{p}$ with radius $r$, i.e. $B(\tilde{p},r) ={\SB}(\tilde{p},r) $.

More generally, if $\tilde{p}$ is a conic point with angle $k\pi$ and the ball $B(\tilde{p},r)$ does not 
contain any other conic points, then $B(\tilde{p},r)$ is the union of $k$ semi-balls of radius $r$, 
centered at $\tilde{p}$ and at the conic points identified with $\tilde{p}$.

For example, if 
$$
\tilde{p} \in \pi_{\mathcal{P}}^{-1}(p)=\{\tilde{p}^1,\ldots,\tilde{p}^k\},
$$
where each $\tilde{p}^i$ is a preimage of a conic point $p$ of angle $k\pi$, then
$$
B(\tilde{p},r) = {\SB}(\tilde{p}^1,r) \cup \cdots \cup {\SB}(\tilde{p}^k,r),
$$
where ${\SB}(\tilde{p}^i,r)$ denotes the semi-ball of radius $r$ centered at $\tilde{p}^i$.

\subsubsection{Accumulation case without intersection with other sides.} First, consider the case where the ball is centered at an accumulation point. In this analysis, it is necessary to assume that at least one side of the multipolygon has an infinite number of identifications, following the Type $W$ identifications.

Let $\tilde{p}_{\infty} \in P$ be an accumulation point of conic points lying on the boundary of the multipolygon $P$, and consider the ball $B(\tilde{p}_{\infty},r)$ of radius $r>0$ centered at $\tilde{p}_{\infty}$. Assume that this ball does not intersect any other edge of the multipolygon. In this situation, $B(\tilde{p}_{\infty},r)$ decomposes as a union of semi-balls measured in the $l_{\infty}$-metric.

 % Let $ B(\tilde{p}_{\infty}, r) $ be a ball centered at the accumulation point that does not intersect any other edge of the multipolygon. In this case, the balls are unions of semi-balls in the $ l_\infty $-metric.  

It is important to note that such a ball contains infinitely many (preimages of) conic points, as its center is itself an accumulation point of conic points. To describe its structure accurately, these conic points must be explicitly taken into account in its construction.

The approach is as follows: we begin by examining all conic points that are within the first semi-ball centered at the accumulation point. For each of these conic points, we compute their distance to the accumulation point. Then, at each conic point (in the same equivalence class) we construct a semi-ball whose radius corresponds to the remaining distance from the accumulation point. This iterative process is continued, accounting for the propagation of identifications through the Type~$ W $ identifications. In doing so, we ensure that the full geometry of the ball $ B(\tilde{p}_{\infty}, r) $ is recovered, including contributions from all relevant conic identifications.

Therefore, the preimage of the ball is given by
$$
B(\tilde{p}_{\infty}, r) = \bigcup_{i = s+1}^\infty B_i,
$$
where $ B_i = {\SB}(\tilde{p}_ir_i) $ for $ i \in \{s+1, s+2, \dots\} $. Here, $ \tilde{p}_{s+1}' $ denotes the endpoint of the segment $ \alpha_s' $, which the ball $ B(\tilde{p}_{\infty}, r) $ touches. The point $ \tilde{p}_i$ represents the endpoint of the segment $ \alpha_i' $, and the radius $ r_i $ of each semi-ball $ B_i $ is given by
$$
r_i = r - \sum_{n=i+1}^\infty a_n,
$$
where $ \{a_n\} $ is the sequence of segment lengths from the Type~$ W $ identification.

\begin{figure}[!h]
    \centering
    \includegraphics[scale=7]{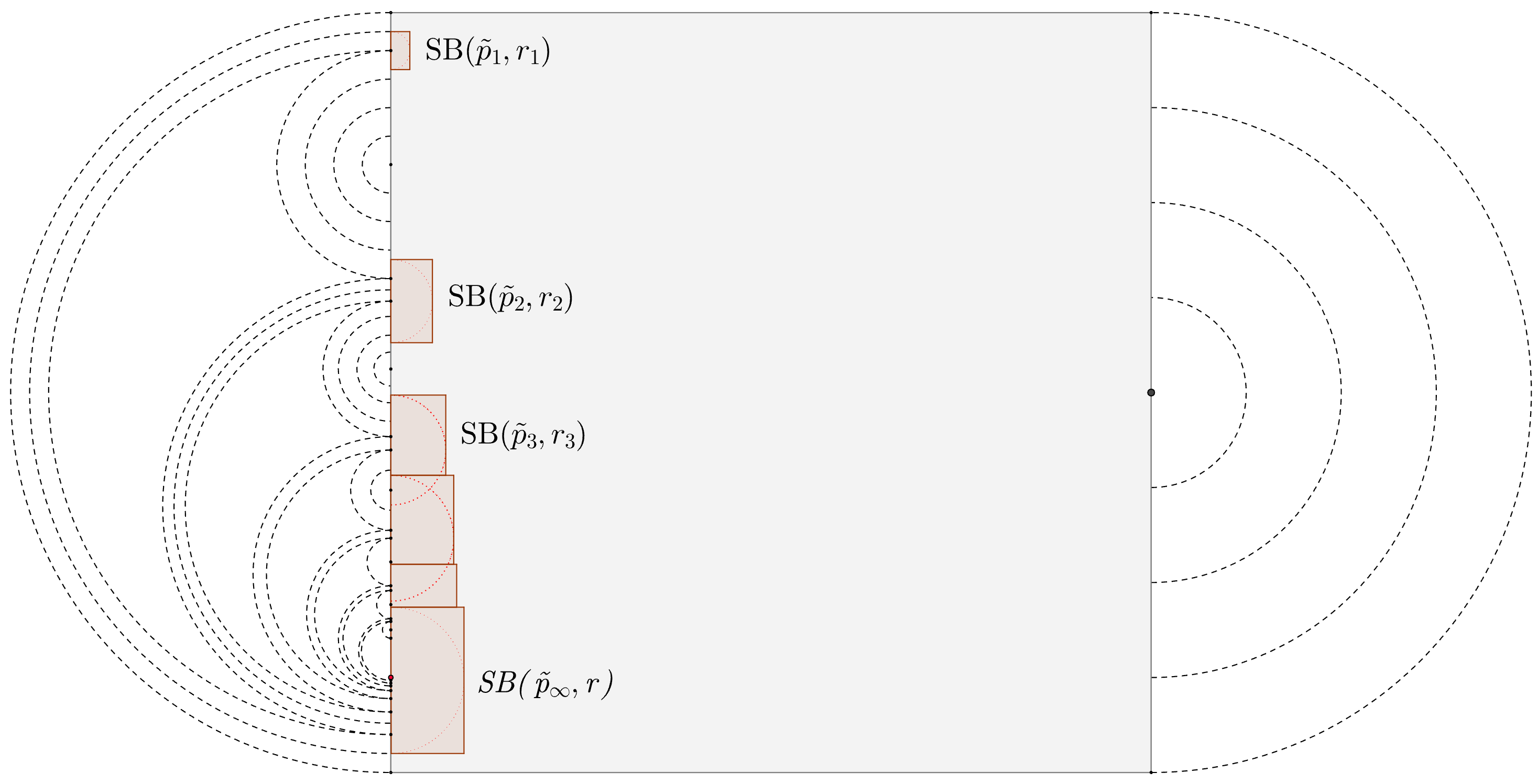}
    \caption{Example of a ball centered at the accumulation point and that does not intersect any other edge of the square.}
    \label{figexemplo}
\end{figure}

Observe that the points $ \tilde{p}_i' $, $ \tilde{p}_i$, and $ \tilde{q}_i $ all arise from the Type~$ W $ identification pattern. More specifically, the points $ \tilde{p}_i$ are the endpoints of the segments $ \beta_i $, the points $ \tilde{q}_i $ are the starting points of the segments $ \alpha_i $, and the points $ \tilde{p}_i' $ are the endpoints of the segments $ \alpha_i' $, which are paired with the $ \alpha_i $ segments. 

Although each $ \tilde{p}_i' $ is identified with both $ \tilde{p}_i$ and $ \tilde{q}_i $, it is sufficient to consider only the semi-balls centered at the points $ \tilde{p}_i$. This is because the radii satisfy $ r_i > r_{i+1} $, and therefore,
$$
\operatorname{SB}(\tilde{p}_{i+1}, r_{i+1}) \supset \operatorname{SB}(\tilde{q}_i, r_i), \quad \text{for } i \in \{s+1, s+2, \dots\}.
$$
This containment ensures that the contribution from semi-balls centered at $ \tilde{q}_i $ is already included within the larger semi-balls centered at $ \tilde{p}_{i+1} $, simplifying the geometric analysis.

\subsubsection{Conic case without intersection with other sides.} Let $p \in \mathcal{L}$ be a conic point on the quotient surface $\mathcal{L}$ and let $\tilde{p} \in P$ be one of its preimages under the projection map $\pi_{\mathcal{P}}\colon P \to \mathcal{L}$. We denote by $B(\tilde{p},r)$ the ball of radius $r>0$ centered at $\tilde{p}$ in the metric of $P$. In this context, $\tilde{p}$ lies on the boundary of $P$ and corresponds, via $\pi_{\mathcal{P}}$, to the conic point $p$ on the quotient surface.%Let $ B(\tilde{p}, r) $ denote a ball centered at a conic point $ p $ with radius $ r $, where $ \tilde{p} $ is one of the preimages in the boundary of $ P $ corresponding to a conic point on the quotient surface $ \mathcal{L} $.

%In this case, it follows that the preimages of these balls under the projection $ \pi_\mathcal{P}: P \to S_\mathcal{P} $ are unions of semi-balls in the $ l_\infty $-metric.
For this construction, it is necessary to assume that certain sides of the multipolygon are assigned identifications either of Type~$W$ or finite basic segment pairings. Hence, without loss of generality, we may consider all types of pairings. We further assume that the ball does not intersect any of the remaining sides of the multipolygon.

As previously observed, if $\tilde{p}$ is a conic point of angle $k\pi$ and the ball 
$B(\tilde{p},r)$ contains no other conic points, then $B(\tilde{p},r)$ decomposes as the union 
of $k$ semi-balls of radius $r$, centered at the preimages of $\tilde{p}$.

%Note that if the radius $ r $ is sufficiently small and $ p $ is a conic point with angle $ \pi $, such that $ B(\tilde{p}, r) $ does not contain other conic points with different angles, then $ B(\tilde{p}, r) $ is a \emph{semi-ball} centered at $ p $ with radius $ r $.

%More generally, if $ p $ is a conic point with angle $ k\pi $, and the ball $ B(\tilde{p}, r) $ does not contain any other conic point, then $ B(\tilde{p}, r) $ is the union of $ k $ semi-balls of radius $ r $, centered at $ p $ and at the conic points identified with $ p $.

If $ \tilde{p} $ is a conic point with angle $ k\pi $, and the ball $ B(\tilde{p}, r) $ contains other conic points, then $ B(\tilde{p}, r) $ is the union of $ k $ semi-balls of radius $ r $, centered at $ \tilde{p} $ and at the points identified with $ \tilde{p} $, together with the semi-balls centered at the additional conic points lying inside these initial semi-balls. The radius of each of these additional semi-balls is given by $ r $ minus the distance from their centers to the nearest previously included conic point that serves as the center of a ball.
This case represents a finite and simplified version of the situation where a ball is centered at a conic point located along a side with Type~$ W $ identifications, which will be discussed in full detail below.

Now, suppose that at least one side of the multipolygon has Type~$W$ identifications, and that the ball $ B(\tilde{p}, r) $ centered at a conic point $\tilde{p}$ contains other conic points. Then the \textit{main ball} (the first semi-ball considered) necessarily contains at least one conic point with angle $ 3\pi $. Consequently, there exists a ball centered either at the endpoint of some segment $ \alpha'_s $, corresponding to the conic point with angle $ 3\pi $ closest to the center $ \tilde{p} $, or at $ \tilde{p} $ itself, in the case where $ \tilde{p} $ is a conic point with angle $ {3\pi} $. 

Assume that this corresponding ball is centered at the point $ \tilde{p}'_{s+1} $, and that its lower boundary intersects the segment $ \alpha'_n $.

 \begin{figure}[!h]
    \centering
    \includegraphics[width=0.8
    \textwidth]{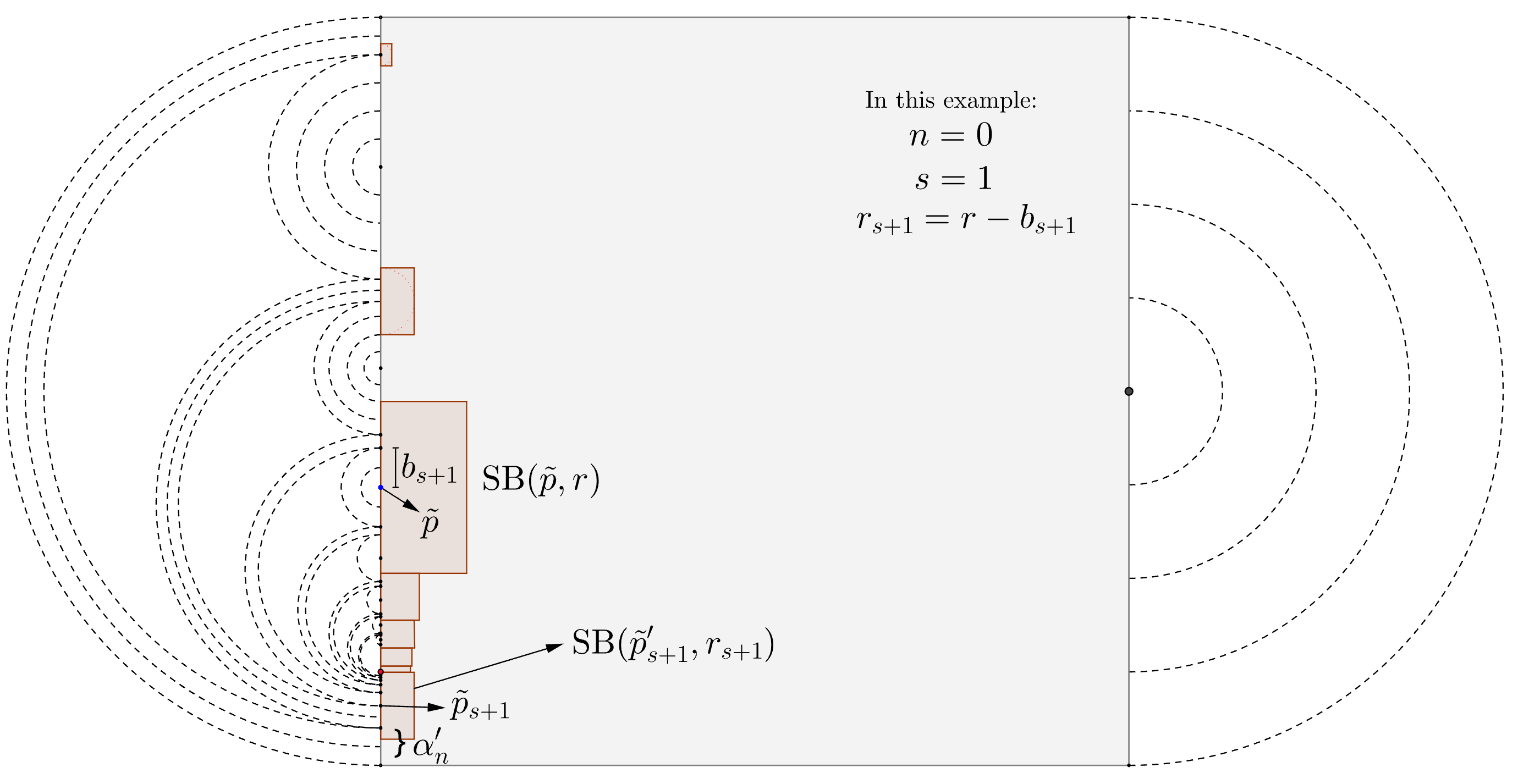}\caption{Ball centered at a conic point.}
    \label{fig8}
\end{figure}

First, we need to construct the semi-balls whose centers correspond to the identifications of $\tilde{p}_j'$ for $j\in\{n+1, \cdots, s\}$. Observe that if $ \tilde{p} $ is a conic point with an angle of $ {3\pi} $, then $ r = r_{s+1} $; otherwise, $ r_{s+1} = r - b_{s+1} $ (see Figure~\ref{fig8}).
Furthermore, note that, for each $j$, $\tilde{p}_j'$ is identified with the points $\tilde{p}_j$ and $\tilde{q}_j$. At each of these points, for $j\in\{n+1, \cdots, s\}$ we have a semi-ball whose radius $r_j$ is given by
    $$r_j=r_{s+1}-\sum_{i=j}^sa_i.$$
    
    If $B(\tilde{p}_{s+1}',r_{s+1})$ contains the accumulation point, then for $j\in \{s+2,\cdots, \infty\}$, we have a semi-ball whose radius $r_j$ is given by
    $$r_j=r_{s+1}-\sum_{i=s+1}^{j-1}a_i.$$
    Therefore, we have
    $$B(\tilde{p}_{s+1}',r_{s+1})= \SB(\tilde{p}'_{s+1}, r_{s+1}) \cup\bigcup^s_{i=n+1}\SB(\tilde{p}_i,r_i) \cup \bigcup^s_{i=n+1}\SB(\tilde{q}_i,r_i) \cup \bigcup^\infty_{i=s+2}\SB(\tilde{p}_i,r_i) \cup \bigcup^\infty_{i=s+2}\SB(\tilde{q}_i,r_i).$$

    Since $r_i>r_{i+1}$ and $\SB(\tilde{p}_{i+1},r_{i+1})\supset \SB(\tilde{q}_i,r_i)$, $j\in\{n+1, \cdots, s\}$, it follows that
    $$B(\tilde{p}_{s+1}',r_{s+1})= \SB(\tilde{p}'_{s+1}, r_{s+1}) \cup\bigcup^s_{i=n+1}\SB(\tilde{p}_i,r_i) \cup  \bigcup^\infty_{i=s+2}\SB(\tilde{p}_i,r_i) \cup \bigcup^\infty_{i=s+2}\SB(\tilde{q}_i,r_i).$$

    On the other hand, for $j\in \{s+2,\cdots, \infty\}$, 
 $\SB(\tilde{q}_{j+1},r_{j+1})\supset \SB(\tilde{p}_{j},r_{j})$, then

    $$B(\tilde{p}_{s+1}',r_{s+1})= \SB(\tilde{p}'_{s+1}, r_{s+1}) \cup\bigcup^s_{i=n+1}\SB(\tilde{p}_i,r_i) \cup  \bigcup^\infty_{i=s+2}\SB(\tilde{q}_i,r_i).$$

    If the ball $B(\tilde{p}_{s+1}',r_{s+1})$ does not contain the accumulation point, then it intersects some $\alpha_t'$. In this way, we have

     $$B(\tilde{p}_{s+1}',r_{s+1})= \SB(\tilde{p}'_{s+1}, r_{s+1}) \cup \bigcup^s_{i=n+1}\SB(\tilde{p}_i,r_i) \cup  \bigcup^t_{i=s+2}\SB(\tilde{q}_i,r_i).$$

Thus, the ball centered at $ \tilde{p} $ with radius $ r $ is given by the union of the half-ball centered at $ \tilde{p} $ with radius $ r $ ( as $ \SB(\tilde{p},r) $), and the ball $ B(\tilde{p}'_{s+1}, r_{s+1}) $. That is,  
$$
B(\tilde{p}, r) = \SB(\tilde{p},r) \cup B(\tilde{p}'_{s+1}, r_{s+1}).
$$

Observe that, so far, we have described balls centered at conic points lying on sides with general basic segment pairings and Type~$W$ identifications, leaving only the case of \emph{finite Type~$W$ identifications}. These can be viewed as particular instances of general basic segment pairings and serve as a useful stepping stone for understanding the construction of balls in more complex settings.

This case is analogous to the Type~$W$ identifications, with the key difference that the resulting ball is always a finite union of semi-balls.

For the description given above, assuming that for Type~$W$ identifications we have $ n_F $ values of $ a_i $ and $ b_i $, that is, $ i \in \{1,2, \dots, n_F\} $, the case where there are infinitely many identifications should be interpreted as $ n_F $, and when there is an accumulation point, it should be read as $ \tilde{p}_{n_F}' $.%, which is also identified with $ q_{n_F} $ and $ p_{n_F} $. Therefore, 

$$
B(\tilde{p}, r) = \SB(\tilde{p},r) \cup B(\tilde{p}'_{s+1}, r_{s+1}).
$$

\subsubsection{Regular case without intersection with other sides.} Now, let $ B(\tilde{p}, r) $ be a ball centered at a (preimage of) regular point that does not intersect two edges of the multipolygon. Furthermore, suppose that if a ball intersects a side, the side has an infinite number of identifications, meaning it follows a Type $W$ of identification.  

It is sufficient to consider this case, as the other types are analogous. In this setting, the  balls are either entire balls or unions of ball fragments and semi-balls in the $ l_\infty $-metric.

There are several possibilities for the preimages of the balls (see examples in Figure \ref{figexemplo3}). If the radius of the ball is sufficiently small such that it contains no non-regular points, then it remains a Euclidean square.

    \begin{figure}[!h]
         \centering
           \includegraphics[scale=6]{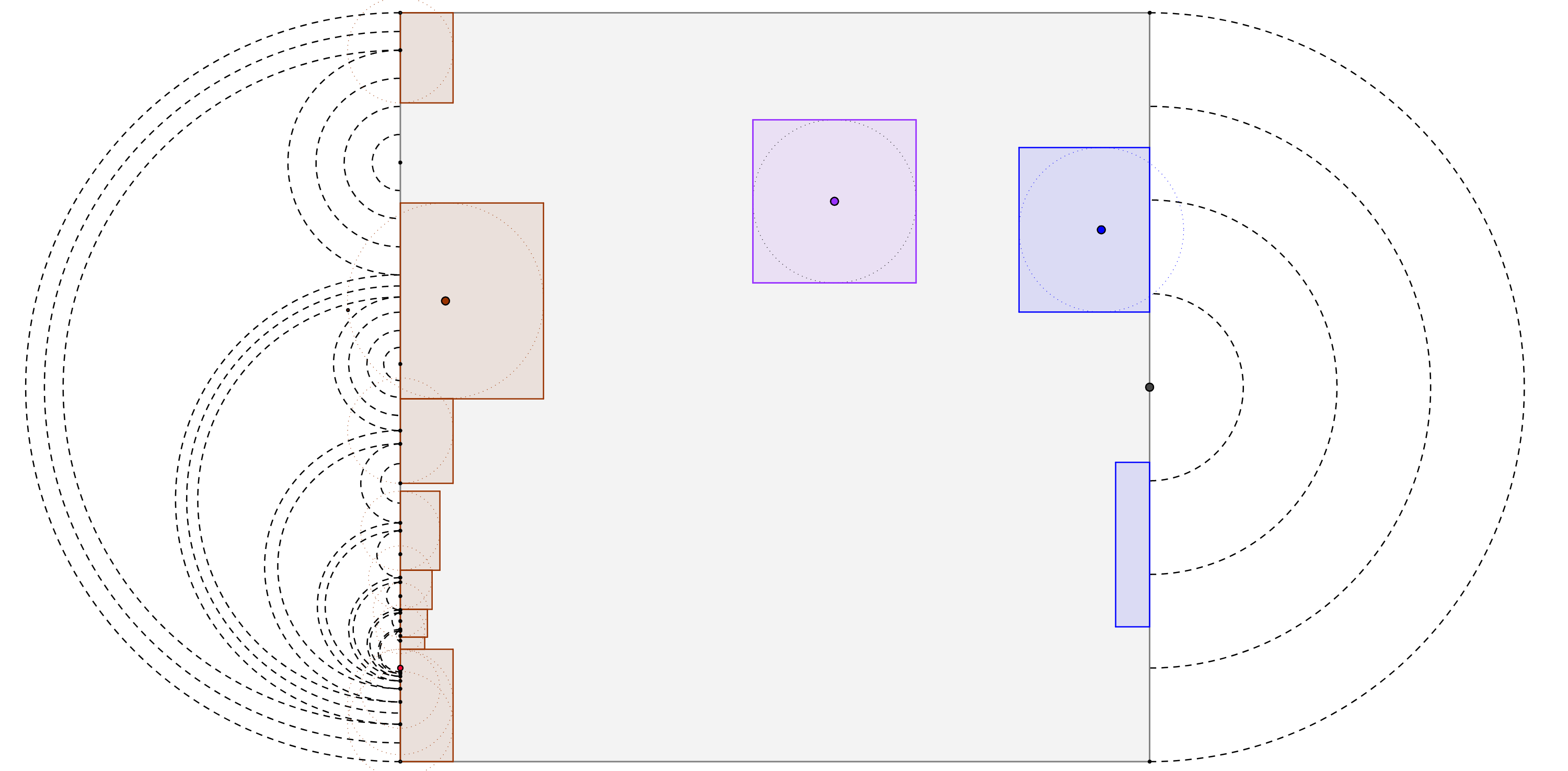}
           \caption{Examples of balls centered at regular points. }
           \label{figexemplo3}
       \end{figure}

In the case where the radius is not small, meaning the ball contains non-regular points, it is necessary to calculate the distance from the center of the ball to each non-regular point. New balls are then constructed with centers at the conic points identified within the main ball. The radii of these balls are determined by subtracting the distance between the corresponding conic point and the center of the initial ball from the radius of the main ball. After this step, the analysis reduces to the previously described conic case, both for intersections with the boundary of Type $W$ identifications and basic segment pairings.

Thus, $ B(\tilde{p}, r) $ is given by the union of the piece of the ball centered at $ \tilde{p} $ with radius $ r $, denoted as $ \PB(\tilde{p},r) $, and a ball centered at a conic point, $ B(\tilde{c},r) $, where $ \tilde{c} $ is either $ \tilde{p}_i $, $ \tilde{p}_i' $, or $ \tilde{q}_i $. It is necessary to consider the largest ball, that is, the one with the greatest radius. The construction of the ball centered at the conic point follows the same procedure as before.

Observe that the regular point may lie on the boundary of the square, specifically on the segments with identifications. In this case, the nearest conic point must be found, and the ball should be constructed accordingly, with $ \PB(\tilde{p},r) = \SB(\tilde{p},r) $. 

Therefore, in all cases where the ball contains at least one conic point, the ball $ B(\tilde{p},r) $ is given by  
$$
B(\tilde{p},r) = \PB(\tilde{p},r) \cup B(\tilde{c},r).
$$

\subsubsection{Balls intersecting other sides.} Now, consider the case where a ball intersects more than one side of the multipolygon. First, observe that if the ball is centered at an accumulation point, then, being a union of smaller balls, its intersection with adjacent sides may lead to a variety of configurations. 

To simplify the exposition, we first describe the construction in the case where the multipolygon $ P $ is a square. The general case, where $ P $ is an arbitrary multipolygon, follows by an inductive procedure over its sides and will be discussed at the end of this section.

Several illustrative examples are presented below to highlight these possibilities.

\begin{enumerate}
\item \textit{If the adjacent side is identified with the opposite side} (see Figure \ref{figsupL}), then  
$$
B(\tilde{p}_{\infty}, r) = \bigcup_1 B_i \cup B(\tilde{p}_0', r_0),
$$
which falls into the previously analyzed case, where $ B(\tilde{p}_0', r_0) $ is the union of two quarter-balls centered at $ \tilde{p}_0' $ and $ \tilde{p}_0 $, each with radius $ r_0 $.

    \item \textit{If the adjacent side follows a folding identification}, then the ball structure changes. In this case,  
    $$
    B(\tilde{p}_{\infty}, r) = \bigcup_1 B_i \cup B(\tilde{p}_0', r_0) \cup B(\tilde{\bar{p}}_0', r_0),
    $$
   where $ B(\tilde{p}_0', r_0) $ is the union of two quarter-balls centered at $ \tilde{p}_0' $ and $ \tilde{p}_0 $, each with radius $ r_0 $, and $ B(\tilde{\bar{p}}_0', r_0) $ is an additional ball that represents a quarter of the ball with radius $ r_0 $, centered at the vertex of the square identified with $ \tilde{p}_0' $. This new ball $B(\tilde{\bar{p}}_0', r_0) $ touches another adjacent side, which may have identifications of Type $W$ or basic segment pairings. Consequently, the type of identification on this newly touched side must also be analyzed, but the process follows the same steps as for the first adjacent side.  

For instance, if the third side follows a folding identification, an additional quarter-ball will appear at the fourth vertex of the square, denoted $ \tilde{\bar{p}}_0 $. The identification type on this fourth side must then be analyzed. If it also follows a folding identification, the quarter-ball centered at $ \tilde{p}_0 $ has already been accounted for.  

On the other hand, if the fourth side has a finite or infinite number of identifications, as in Types $W$ and finite $W$, then instead of a single quarter-ball, $ B(\tilde{\bar{p}}_0, r_0) $ consists of either a finite or infinite union of balls, respectively, centered at conic points. In this case, the structure falls into the conic case.

Therefore, we have  
$$
B(\tilde{p}_{\infty}, r) = \bigcup_1 B_i \cup B(\tilde{p}_0', r_0) \cup B(\tilde{\bar{p}}_0', r_0) \cup B(\tilde{\bar{p}}_0, r_0),
$$
where the balls $ B(\tilde{p}_0', r_0) $, $ B(\tilde{\bar{p}}_0', r_0) $, and $ B(\tilde{\bar{p}}_0, r_0) $ depend on the type of identification of the sides they intersect in a counterclockwise direction. That is, they may be either a finite or infinite union of balls.

    \item \textit{If the adjacent side has an infinite number of identifications (Type $W$)}, then the point $ \tilde{p}_0' $ is considered a conic point for this side. Consequently, the ball centered at this point satisfies  
    $$
    B(\tilde{p}_{\infty}, r) = \bigcup_1 B_i\cup  B(\tilde{p}_0', r_0) \cup B(\tilde{\bar{p}}_0', r_0),
    $$
    meaning it falls into the conic case. As before, this new ball $ B(\tilde{\bar{p}}_0', r_0) $ touches another adjacent side, which may have identifications of Type $W$ or basic segment pairings. The identification type of this new side must also be analyzed, following the same steps applied to the first adjacent side.

    Therefore, we have  
$$
B(\tilde{p}_{\infty}, r) = \bigcup_1 B_i \cup B(\tilde{p}_0', r_0) \cup B(\tilde{\bar{p}}_0', r_0) \cup B(\tilde{\bar{p}}_0, r_0),
$$
where the balls $ B(\tilde{p}_0', r_0) $, $ B(\tilde{\bar{p}}_0', r_0) $, and $ B(\tilde{\bar{p}}_0, r_0) $ depend on the type of identification of the sides they intersect in a counterclockwise direction. That is, they may be either a finite or infinite union of balls.

    \item \textit{If the adjacent side has a finite type $W$ identifications}, then $ \tilde{p}_0' $ is considered a conic point for this side. Consequently, constructing the ball centered at this point satisfies  
    $$
    B(\tilde{p}_{\infty}, r) = \bigcup_1 B_i \cup  B(\tilde{p}_0', r_0)\cup B(\tilde{\bar{p}}_0', r_0),
    $$
    which places it in the conic case for balls that do not contain the accumulation point. Again, this new ball $ B(\tilde{\bar{p}}_0', r_0) $ touches another adjacent side, which may have identifications of Type $W$ or basic segment pairings. The identification type of this adjacent side must be analyzed, but the process remains the same as for the first adjacent side.

    Therefore, we have  
$$
B(\tilde{p}_{\infty}, r) = \bigcup_1 B_i \cup B(\tilde{p}_0', r_0) \cup B(\tilde{\bar{p}}_0', r_0) \cup B(\tilde{\bar{p}}_0, r_0),
$$
where the balls $ B(\tilde{p}_0', r_0) $, $ B(\tilde{\bar{p}}_0', r_0) $, and $ B(\tilde{\bar{p}}_0, r_0) $ depend on the type of identification of the sides they intersect in a counterclockwise direction. That is, they may be either a finite or infinite union of balls.
\end{enumerate}

Now, observe that for cases where the balls are centered at either regular points or conic points, the analysis is analogous to that for accumulation points. This is because the balls intersecting other sides of the square are those centered at its vertices, which, depending on the type of boundary identification, may be the union of two quarter-balls, a finite union of semi-balls, or an infinite union of semi-balls.  

Thus, in the end, we obtain  
$$
 B(\tilde{p}, r) = \hat{B}(\tilde{p},r) \cup B(\tilde{p}_0', r_0) \cup B(\tilde{\bar{p}}_0', r_0) \cup B(\tilde{\bar{p}}_0, r_0),
$$
where $ \hat{B}(\tilde{p},r) $ is the ball constructed on the first side without considering the balls centered at the vertices, and the balls $ B(\tilde{p}_0', r_0) $, $ B(\tilde{\bar{p}}_0', r_0) $, and $ B(\tilde{\bar{p}}_0, r_0) $ are those centered at the vertices, whose structure depends on the type of identification of the sides they intersect in a counterclockwise direction.

Recall that if the first adjacent side identification is of folding identification, this union simplifies to  
$$
 B(\tilde{p}, r) = \hat{B}(\tilde{p},r) \cup B(\tilde{p}_0', r_0).
$$

We emphasize that the number of possible identifications along the sides of the square is infinite. The four types discussed here were chosen for illustrative purposes, and to provide a conceptual framework for the construction. In general, the geometry of the ball depends intricately on the nature of the pairings prescribed along the boundary.

We now turn to the general case, where $ P $ is a multipolygon. The goal is to extend the previous construction, developed for the square, to situations where the ball intersects two or more sides of a general multipolygon.

Suppose that a ball $ B(\tilde{p}, r) $, centered at a point $ \tilde{p} \in P $, intersects at least two sides of the multipolygon. We begin by constructing the portion of the ball that intersects the first side, following the same procedure as before for the case in which the ball does not intersect another side. We then analyze the portion that intersects the second side, applying the same method.

Depending on the type of identification along each side and on the number of conic points contained within the intersecting region, a finite or infinite number of additional parts of the ball (such as semi-balls or quarter-balls) may arise, centered at the points identified with the conic points contained in the ball.

These new balls may themselves intersect further sides of the polygon. When that occurs, the process is repeated: we examine whether the newly intersected region contains conic points, and if so, we construct additional local balls at the corresponding identified points. This step may in turn lead to intersections with further adjacent sides, which are then handled recursively.

This yields an inductive construction over the sides of the multipolygon. At each step:
\begin{itemize}
    \item we identify which sides are intersected by the current ball configuration;
    \item for each such side, we check whether the intersection contains conic points;
    \item if it does, we construct new parts of the ball centered at the corresponding identified points.
\end{itemize}

This inductive procedure continues until all sides intersected by the original ball $ B(\tilde{p}, r) $ have been considered. If the ball intersects only two sides, the process stops after two steps; if it intersects three, it proceeds to the third; and in general, it terminates in at most $ n $ steps, where $ n $ is the number of sides of the multipolygon.

At each step, the construction of the parts of the ball follows exactly the same principles described in the square case, depending on the type of identification on each side and on the nature (regular, conic, or accumulation) of the intersected points.

This method offers a clear framework for analyzing the local geometry of metric balls in paper spaces constructed from arbitrary side identifications on multipolygons.

Importantly, the resulting shape and complexity of a ball are influenced not just by how many sides it intersects but also by the nature of the identifications involved - be they foldings, finite patterns, or infinite configurations like those of Type~$W$. Each identification gives rise to new conic points or gluing structures, which in turn determine how the ball extends across the surface.

By advancing through the sides in a step-by-step manner, the construction ensures that all relevant geometric contributions are accounted for and incorporated into the final structure of the ball.

%This method provides not only a concrete description of the ball’s geometry but also a framework for analyzing more subtle properties of the space, such as curvature concentration, singularities, and local isometries. It highlights how the interplay between discrete identifications and continuous metric structures gives rise to rich geometric behavior, even in surfaces built from simple polygonal pieces.

%CONCLUSAO Note that, for a ball centered at a conic point with an angle of $\pi$ or $3\pi$, we construct additional balls at the points identified with the conic points contained within the original ball. The radii of these new balls are determined by subtracting the distance between the corresponding conic point and the center of the initial ball from the radius of the original ball.

\subsection{Proof of Theorem \ref{teo_principal}}

Initially, in this section, $(P,\mathcal{P})$ is again a paper-folding scheme with associated paper sphere or paper space $\mathcal{L}_\mathcal{P}\subset \mathcal{L}$.  We show that the surfaces in the class $ \mathcal{L} $ satisfy the conditions of being  Ahlfors 2-regular and linearly locally contractible. As a consequence, we conclude that the surfaces of type $ \mathcal{L}^* $ are quasisymmetrically equivalent to the 2-sphere.

\subsubsection{Ahlfors 2-regularity}
%In this section, we show that the class of surfaces described above satisfies the conditions of being Ahlfors 2-regular and linearly locally connected (LLC). As a consequence, these surfaces are quasisymmetrically equivalent to the 2-sphere.
In order to prove that a space $ \mathcal{L}_\mathcal{P} $ is Ahlfors $ 2 $-regular, we first need to establish certain properties of the metric structure on $ \mathcal{L}_\mathcal{P}  $. Specifically, we need to show that there exists a constant $ r_0 > 0 $ such that the conditions of Ahlfors regularity are satisfied for $ \mathcal{L}_\mathcal{P}  $. To achieve this, we will apply Lemma~\ref{lemma-ahlfors}, which provides a criterion for Ahlfors regularity in terms of the measure of balls in a metric space. 

%The following lemma will be crucial in demonstrating that $ L $ satisfies the necessary regularity conditions by determining an appropriate value of $ r_0 $ for the given space.

\begin{lemma}\label{lemma-ahlfors}
Let $ X $ be a metric space such that $ \mathcal{H}^Q(X) < \infty $. Suppose there exist positive constants $ r_0 > 0 $ and $ c_0 > 0 $ such that for all $ 0 < r \leq r_0 $ and for every $ a \in X $, we have
$$
\frac{1}{c_0} r^Q \leq \mathcal{H}^Q(\overline{B}(a, r)) \leq c_0 r^Q.
$$
Then $ X $ is Ahlfors $ Q $-regular.
\end{lemma}

\begin{proof}
Let $ C := \max\left\{ c_0, \frac{\mathcal{H}^Q(X)}{r_0^Q} \right\} $, and let $ r > 0 $ be arbitrary.

If $ 0 < r \leq r_0 $, the assumption gives directly
$$
\frac{1}{c_0} r^Q \leq \mathcal{H}^Q(\overline{B}(a, r)) \leq c_0 r^Q \leq C r^Q.
$$

Now, suppose $ r > r_0 $. Then
$$
\mathcal{H}^Q(\overline{B}(a, r)) \leq \mathcal{H}^Q(X) = \frac{\mathcal{H}^Q(X)}{r^Q} \cdot r^Q \leq \frac{\mathcal{H}^Q(X)}{r_0^Q} \cdot r^Q \leq C r^Q.
$$

For the lower bound, note that $ \overline{B}(a, r) $ contains a ball of radius $ r_0 $, since $ r > r_0 $, so
$$
\mathcal{H}^Q(\overline{B}(a, r)) \geq \mathcal{H}^Q(\overline{B}(a, r_0)) \geq \frac{1}{c_0} r_0^Q =: c_1 r^Q,
$$
with $ c_1 = \frac{1}{c_0} \cdot \frac{r_0^Q}{r^Q} \geq \frac{1}{c_0} \cdot \frac{r_0^Q}{r^Q} $. Then for $ r \geq r_0 $, this implies
$$
\mathcal{H}^Q(\overline{B}(a, r)) \geq \frac{r^Q}{c_0 \cdot (r/r_0)^Q} = \frac{1}{c_0} r_0^Q.
$$

Therefore, $ X $ is Ahlfors $ Q $-regular.
\end{proof}

\emph{Proof of Theorem \ref{teo_principal}(a).} Since the preimages of balls in $\mathcal{L}_\mathcal{P}$ under the projection map are subsets of $ \mathbb{R}^2 $, the Hausdorff 2-measure in this context is given by a constant factor that multiplies the area of the ball.

Furthermore, different types of balls must be considered depending on the location of their centers. Specifically, three cases arise: balls centered at conic points, at the accumulation point, and at regular points. A detailed examination of each case is therefore necessary.

In order to apply Lemma~\ref{lemma-ahlfors} and show that a space $ \mathcal{L}_\mathcal{P}$ is Ahlfors $ 2 $-regular, we first need to determine an appropriate value for $ r_0 $. Consider a multipolygon $ P $ with $ n $ edges. Denote the sides of $ P $ by $ e_1, e_2, \dots, e_n $. Let $ d_{\min} $ represent the smallest distance between non-adjacent sides of $ P $, defined as
$$
d_{\min} = \min \{ d(e_i, e_j) : j \neq i+1 \}.
$$
Note that for $ r < d_{\min}/2 $ and for a point $ \tilde{p} $ located on one of the sides of $ P $, the ball $ B(\tilde{p}, r) $, centered at $ \tilde{p} $ with radius $ r $, does not intersect non-adjacent sides of any polygon. At this stage, we refer to $ B(\tilde{p}, r) $ as the first ball, without considering any identifications on the sides.

For points $ \tilde{p} $ on one of the sides $ e_i $ and for $ r < \text{diam}(P) $, there exists a constant $ K $ such that for $ r / 2K_r < d_{\min}/2 $, the ball $ B(\tilde{p}, r / 2K_r) $ does not intersect any non-adjacent sides. 

It is important to note that $ K_r $, at this point,  depends on the radius $ r $, and we need to find a $ K $ that works for all $ r $. To do so, we consider the largest possible value for $ r $, which is $ r = \text{diam}(P) $. For this case, there exists a constant $ K $ such that $ B(\tilde{p}, r / 2K) $ (before considering the identifications on the edges) does not intersect any non-adjacent side, for each $ \tilde{p} $ on a side of $ P $.

Therefore, we conclude that for any $ r < \text{diam}(P) $, the same $ K $ works, ensuring that the ball $ B(\tilde{p}, r) $ does not intersect any non-adjacent sides of the multipolygon $ P $ (before considering the identifications on the edges). 

We can then define $ r_0 = \frac{\text{diam}(P)}{2K} $ as the desired radius.

%To establish this condition, it is necessary to determine the shape of the metric balls. However, as observed earlier, it suffices to analyze the balls under the projection $\pi_\mathcal{P}$, since the metric quotient respects the identifications along the boundary (see \textcolor{red}{Proposition X}).  

Throughout this proof, any reference to a ball refers to its preimage under the projection map and to a conic point or {conic point of angle $ k\pi $} (with $ k \in \mathbb{N} $), refers to one of its $ k $ preimages at the boundary of $ P $.

The condition of being Ahlfors 2-regular at regular points is straightforward if we consider radii small enough so that the balls are far from the non-regular points. At these points, the balls will be Euclidean squares, and since the $l_\infty$-metric and the Euclidean metric are equivalent, by {Lemma} \ref{equivhausdorff}, there exists a constant $L$ such that
 
  $$\mathcal{H}^2(B(\tilde{p},r))= \frac{1}{4}\pi L^2 \operatorname{Area}(B(\tilde{p},r))= \frac{1}{4} \pi L^2  4r^2=\pi L^2 r^2$$
  which implies the desired bounded.

The condition of Ahlfors 2-regularity at the conic points follows straightforwardly when sufficiently small radii are considered, such that the ball does not contain any other conic points. In this case, we have

$$
\mathcal{H}^2(B(\tilde{p},r)) \leq \frac{1}{4} \pi L^2 k \operatorname{Area}\left( \SB(\tilde{p}, r) \right) = \frac{1}{4} \pi L^2 k \cdot 2 r^2 = \frac{1}{2} \pi L^2 k r^2,
$$
where $ \tilde{p} $ is a conic point with angle $ k\pi $.

If at least one side of the multipolygon has Type~$W$ identification, consider $ B(\tilde{p}_{\infty}, r) $ to be a ball centered at the accumulation point that does not intersect any other edge of the multipolygon.

Assume that %the lower end of
the boundary of the ball intersects the segment $\alpha'_n$. The upper bound for the area of the ball $B(\tilde{p}_{\infty},r)$ can be determined as follows: consider the area of the rectangle that contains all the balls centered at the points $\tilde{p}_i$, for $i\in \{n+1, \cdots, \infty\}$, plus the area of the ball centered at the point $\tilde{p}_n$, as illustrated in Figure \ref{Quali_area_OExemplooo}. 

\begin{figure}[!h]
         \centering
           \includegraphics[scale=6]{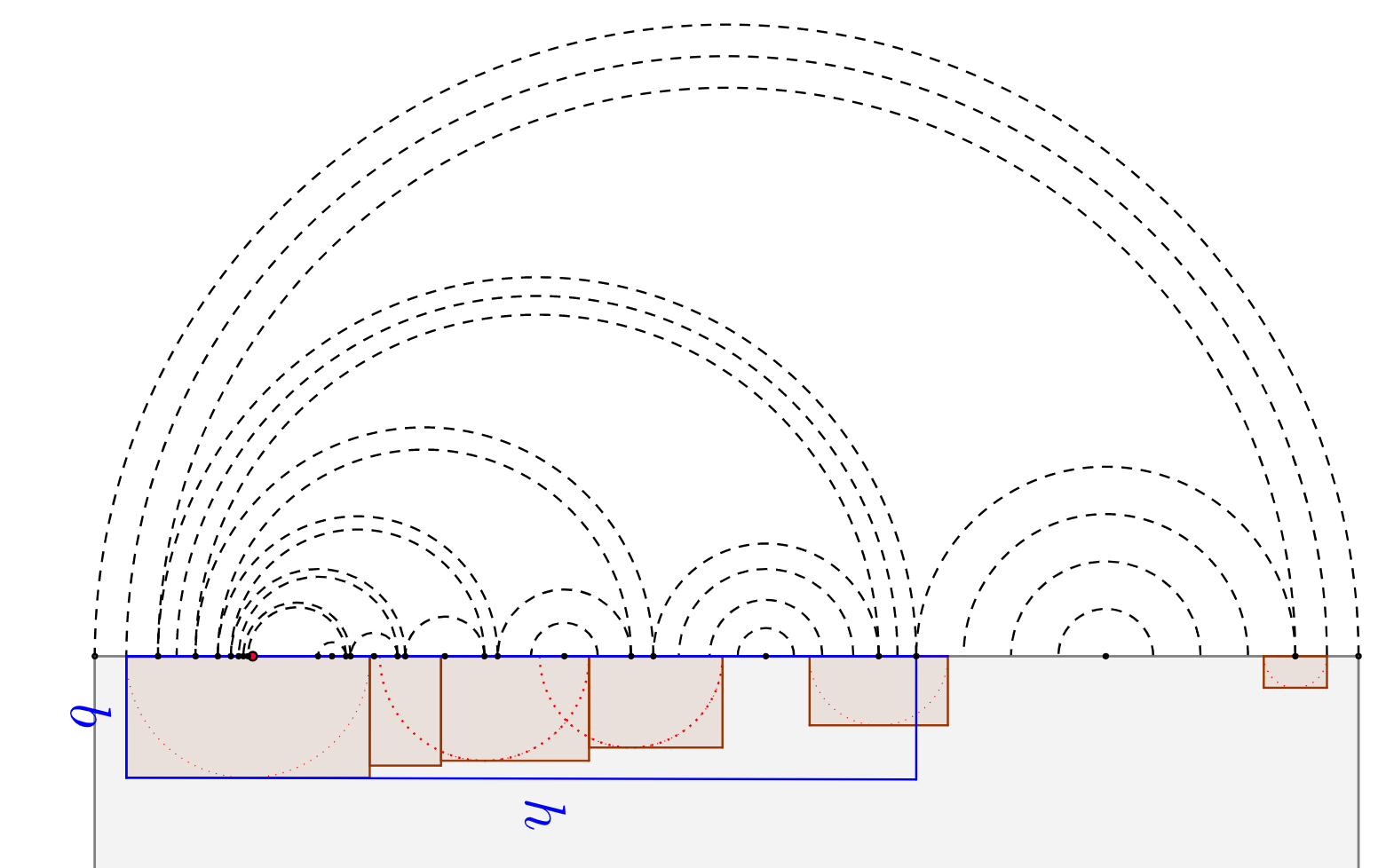}
           \caption{Figure to illustrate how the upper bound of the area will be determined.}
           \label{Quali_area_OExemplooo}
       \end{figure}

       Since  $r=\sum_{i=n+1}^\infty a_k+\epsilon a_n$, where  $\epsilon\in\mathbb{R}_+$ and $\epsilon<1$, 
\begin{align*}
    \operatorname{Area}(B(\tilde{p}_{\infty},r_n))\le& \  bh+\operatorname{Area}(B(\tilde{p}_n,\epsilon a_n))+\epsilon a_n r \\
    \le& \  r \left( r+2\sum_{n+1}^\infty b_k+\sum_{n+1}^\infty a_{k}\right)+ \frac{r^2}{2}+r^2\\
\end{align*}

Since $\{a_i\}_{i \in \mathbb{N}}$ and $\{b_i\}_{i \in \mathbb{N}}$ are sequences of constant values satisfying
$$
\sum_{i=0}^\infty a_i + \sum_{i=0}^\infty b_i = \ell/2,
$$
it follows that there exists a constant $K' > 0$ such that for all $n \in \mathbb{N}$,
$$
\sum_{i=n+1}^\infty b_i \le K' \sum_{i=n+1}^\infty a_i.
$$
Therefore, 

\begin{align*}
    \operatorname{Area}(B(\tilde{p}_{\infty},r_n))\le& \  r \left( r+ 2K'\sum_{n+1}^\infty a_{k}+\sum_{n+1}^\infty a_{k}\right)+\frac{3r^2}{2}\\
    \le& \  r \left( r+ (2K'+1)r\right)+\frac{3r^2}{2}\\
     \le& \ \left( 2K'+ \frac{7}{2}\right) r^2.
\end{align*}

In particular, consider the case in which the sequences decrease geometrically; that is, when
$$
a_k = a_0 \lambda^{-k}, \quad b_k = b_0 \lambda^{-k},
$$
for some real number $\lambda > 1$. In this setting, the following inequalities are satisfied. This case is of particular importance, as paper surfaces arising from dynamical systems frequently exhibit this type of behavior (see~\cite{de2004unimodal}).

\begin{align*}
    \operatorname{Area}(B(\tilde{p}_{\infty},r_n))\le& \  bh+\operatorname{Area}(B(\tilde{p}_n,\epsilon a_n))+\epsilon a_n r \\
    \le& \  r \left( r+2\sum_{n+1}^\infty b_k+\sum_{n+1}^\infty a_{k}\right)+ \frac{r^2}{2}+r^2\\
    \le & \ r \left( r+2\left( \frac{b_0\lambda^{-(n+1)}}{1-\lambda^{-1}}\right)+r\right)+\frac{3r^2}{2}\\
    \le & \ r \left(2 r+2\left( \frac{b_0\lambda^{-n}}{1-\lambda^{-1}}\right)\right)+\frac{3r^2}{2}.\\
\end{align*}

If $b_0>a_0$, where  $a_0$ and $b_0$ are fixed real numbers, then there exists a $p\in \mathbb{R}$  such that $pa_0>b_0$. Therefore, 

\begin{align*}
    \operatorname{Area}(B(\tilde{p}_{\infty},r))\le & \ r \left(2 r+2{a_0}{p}\left( \frac{\lambda^{-n}}{1-\lambda^{-1}}\right)\right)+\frac{3r^2}{2}\\
    = & \ r (2r+{2}{p}r)+\frac{3r^2}{2}
    =\left({2}{p}+\frac{7}{2}\right)r^2.
\end{align*}

And, if $a_0>b_0$, we have

\begin{align*}
    \operatorname{Area}(B(\tilde{p}_{\infty},r))\le & \ r \left(2 r+2{a_0}\left( \frac{\lambda^{-n}}{1-\lambda^{-1}}\right)\right)+\frac{3r^2}{2}\\
    = & \ r (2r+{2}r)+\frac{3r^2}{2}\\
    =& \ \frac{11}{2}r^2.\\
\end{align*}

The upper bound for the area of the ball $B(\tilde{p}_{\infty},r)$ can be determined by
$$\operatorname{Area}(B(\tilde{p}_{\infty},r))\ge 2r^2.$$

If at least one side of the multipolygon has Type~$W$ identification, consider  $ B(\tilde{p}, r) $ to be a ball centered at a conic that does not intersect any other edge of the multipolygon. For the general case of Ahlfors 2-regularity at conic points, a similar argument to that used for accumulation points applies.

%Furthermore, suppose the side has an infinite number of identifications, meaning it follows a Type 3 identification.   It is sufficient, now, to consider this case, as the other types are analogous and simpler.

In this case, for any ball centered at a conic point that is above the accumulation point, when we make the identifications with respect to the edge, a ball centered on a conic point below the accumulation point will appear.

We will divide this into two cases:
\begin{enumerate}[C1)]
    \item The ball that is centered at the conic point below the accumulation point  contains it;

    We will denote by $n$ the index corresponding to the lower $\alpha_n'$ that this ball touches and by $k$ the index corresponding to the upper $\alpha_k'$ that this ball touches. This means there is a new ball centered at the point $\tilde{p}_k'$.  Thus,
    $$r>\sum_{i=n+1}^{k-1}a_i \ \text{and}
\ r>\sum_{i=k}^\infty a_i.$$

The total area is given by the area of the ball centered at the point below the accumulation point, plus twice the area of the ball corresponding to $\alpha_n$ (to account for the piece of $\beta_n$ of the ball corresponding to $\alpha_{n+1}$), and the area of two rectangles. The first rectangle, with height $h_1$, contains the balls corresponding to the identifications of the ends of $\alpha_i$ for $i\in \{n+1, \cdots, k-1\}$. The second rectangle, with height $h_2$, corresponds to the main ball and the balls corresponding to the indices $i\in \{k+1, k+2,\cdots \}$.

    Therefore, 
    \begin{align*}
        \operatorname{Area}(B(\tilde{p},r))<&  \ {2}r^2+  r^2+r\cdot h_1 + r\cdot h_2\\
        <& \ 6 r^2+r\left(r+2\sum_{n+1}^{k-1}b_i+\sum_{n+1}^{k-1}a_i\right)+ r\left(r+2\sum_{k+1}^\infty b_i+\sum_{k+1}^\infty a_i\right)\\
        <& \ 6 r^2 + r\left(r+2\sum_{n+1}^{k-1}b_i+r\right)+ r\left(r+2\sum_{k+1}^\infty b_i+r\right).
    \end{align*}
%    If $b_0>a_0$, where  $a_0$ and $b_0$ are fixed real numbers, there will be $p\in \mathbb{R}$  such that $pa_0>b_0$.

Since $\{a_i\}_{i \in \mathbb{N}}$ and $\{b_i\}_{i \in \mathbb{N}}$ are sequences of constant values satisfying
$$
\sum_{i=0}^\infty a_i + \sum_{i=0}^\infty b_i = \ell/2,
$$
it follows that there exists a constant $K' > 0$ such that for all $n \in \mathbb{N}$,
$$
\sum_{i=n+1}^{k-1} b_i \le K'' \sum_{i=n+1}^{k-1}  a_i.
$$
Therefore,

    \begin{align*}
        \operatorname{Area}(B(\tilde{p},r))<& \ 6 r^2 + r(2r +2K''r)+r(2r+2K''r)\\
        =&\left(10+4K''\right)r^2.
    \end{align*}

    %If $a_0>b_0$, we have
    %  \begin{align*}
     %   \operatorname{Area}(B(a,r))<& \ 6 r^2 + r(2r +2r)+r(2r+2r)\\
      %  =&\ 14r^2.
   % \end{align*}

  \begin{figure}[h]
    \centering
    \includegraphics[width=0.45
    \textwidth]{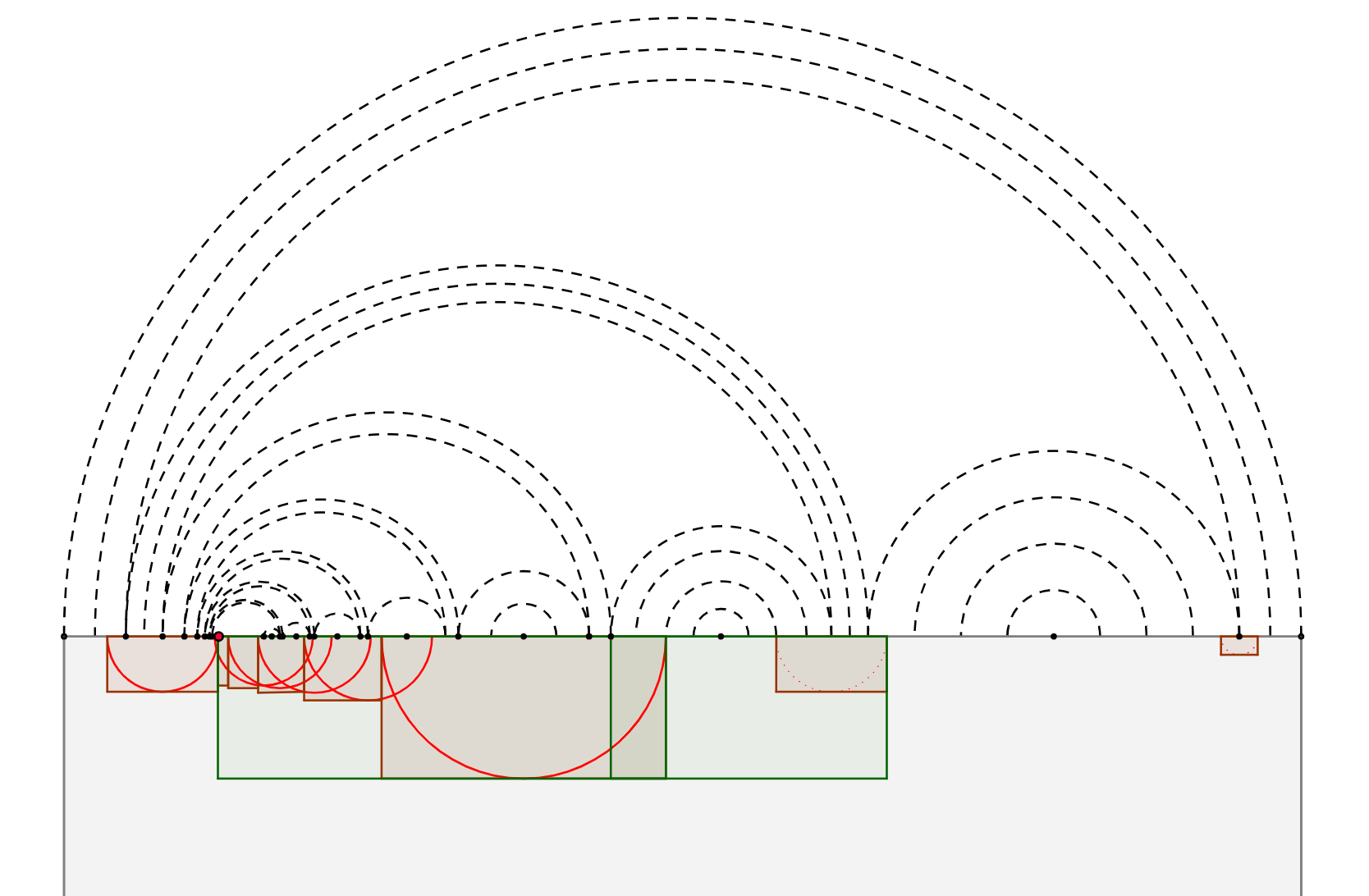}\caption{Figure to illustrate how the upper bound of the area will be determined.  }
\end{figure}

    \item The ball that is centered at the conic point below the accumulation point does not contain it.

    Again, we will denote by $n$ the index corresponding to the lower $\alpha_n$ that this ball touches and by $k$ the index corresponding to the upper $\alpha_k$ that this ball touches. Thus, 
    $$r>\sum_{i=n+1}^{k-1}a_i \ \text{ and } \ 
r>\sum_{i=k}^sa_i.$$

A similar argument to the one used for  case C1) applies:

    \begin{align*}
        \operatorname{Area}(B(\tilde{p},r))<&  {2}r^2+ 4 r^2+r\cdot h_1 + r\cdot h_2\\
        <&6 r^2+r\left(r+2\sum_{n+1}^{k-1}b_i+\sum_{n+1}^{k-1}a_i\right)+ r\left(r+2\sum_{k}^sb_i+\sum_{k}^sa_i\right)\\
        =&6 r^2 + r(2r +2K'''r)+r(2r+2K'''r)\\
        =&\left(10+4K'''\right)r^2
    \end{align*}

\end{enumerate}

Furthermore, for balls centered at the conic points, the lower bound for the area of the balls can be given by 
$$\operatorname{Area}(B(\tilde{p},r))>{2}r^2,$$ 
since at least the semi-ball - before the identifications are made - will be considered.

Consider a ball $ B(\tilde{p}, r) $ centered at a conic point located on a side of the multipolygon, where the side does not have a Type~$W$ identification, and the ball does not intersect any other side of the polygon. In this case, the ball will be a finite union of balls with radii smaller than or equal to $ n_p $. Therefore, we have the following inequality for the area:

$$
2r^2\leq \operatorname{Area}(B(\tilde{p}, r)) \leq 2n_p r^2,
$$
where $ k $ is the angle factor associated with the conic point.

In particular, if there is a finite Type~$W$ identification on a side of the multipolygon, let $ B(\tilde{p}, r) $ be a ball centered at a conic point $ \tilde{p} $ with radius $ r $. In this case, the ball $ B(\tilde{p}, r) $ will be a finite union of semi-balls, each with radius smaller than or equal to $ r $. More precisely, we have the following:

\begin{enumerate}[ C1)]
    \item The ball centered at the conic point below the point $\tilde{p}_{n_F}'$ contains it.

    Observe that, the main ball necessarily contains at least one conic point with an angle of $ 3\pi $, there will be a ball centered at the endpoints of some segment $ \alpha'_s $ corresponding to the conic point with an angle of $ 3\pi $ closest to the center $ \tilde{p} $, or at $ \tilde{p} $ itself if it is a conic point with an angle of $ {3\pi} $, similarly to Section \ref{bolas-sem-ling}. Assume that this corresponding ball is centered at the point $ \tilde{p}'_{s+1} $ and that its lower boundary intersects the segment $ \alpha'_n $. 

   Thus,
    $$
    r > \sum_{i=n+1}^{s} a_i \quad \text{and} \quad r > \sum_{i=s+1}^{n_F} a_i.
    $$
    We know that
    $$
    B(\tilde{p},r) = \SB(\tilde{p},r) \cup B(\tilde{p}_{s+1}', r_{s+1}),
    $$
    where
    $$
    B(\tilde{p}_{s+1}', r_{s+1}) = \bigcup_{i=n+1}^{s} \SB(\tilde{p}_i,r_i) \cup \bigcup_{i=s+2}^{n_F} \SB(\tilde{q}_i,r_i),
    $$
    and the radii are given by:
    $$
    r_j = r_{s+1} - \sum_{i=j}^{s} a_i, \quad \text{for } j \in \{n+1, n+2, \dots, s\},
    $$
    and
    $$
    r_j = r_{s+1} - \sum_{i=s+1}^{j-1} a_i, \quad \text{for } j \in \{k+1, k+2, \dots, n_F\}.
    $$

    Therefore, the area is given by
    $$
    \operatorname{Area}(B(\tilde{p},r)) = 2r^2 + \sum_{i=n+1}^{s} r_i^2 + r_{s+1}^2+\sum_{i=s+2}^{n_F} r_i^2 \le (n_F+1)r^2.
    $$

\item The ball that is centered at the conic point below the point $\tilde{p}_{n_F}'$ does not contain it.

    Again,  there will be a ball centered at the endpoints of some segment $ \alpha'_s $ corresponding to the conic point with an angle of $ 3\pi $ closest to the center $ \tilde{p}$, or at $ \tilde{p} $ itself if it is a conic point with an angle of $ {3\pi} $. Assume that this corresponding ball is centered at the point $ \tilde{p}'_{s+1} $ and that its lower boundary intersects the segment $ \alpha'_n $. Furthermore, denote  by $t$ the index corresponding to the upper $\alpha_l$ that this ball touches.  Thus, 
    $$r>\sum_{i=n+1}^{s}a_i \ \text{ and } \ 
r>\sum_{i=s+1}^ta_i.$$

A similar argument to the one used for  case C1) applies:

$$
    \operatorname{Area}(B(\tilde{p},r)) = 2r^2 + \sum_{i=n+1}^{s} r_i^2 + r_{s+1}^2+\sum_{i=s+2}^{t} r_i^2\le (l+1)r^2 \le (n_F+1)r^2.
    $$
\end{enumerate}

Now, let $ B(\tilde{p}, r) $ be a ball centered at a regular point that does not intersect any other edge of the multipolygon. In this case, it is possible to derive the following inequality for the area:

$$
2r^2 \leq \operatorname{Area}(B(\tilde{p}, r)) \leq 4r^2.
$$

In the case where the radius is not small, meaning the ball contains non-regular points, it is necessary to compute the distance from the center of the ball to each non-regular point. New balls are then constructed with centers at the conic points identified within the main ball. The radii of these balls are determined by subtracting the distance between the corresponding conic point and the center of the initial ball from the radius of the main ball.  

After this step, the analysis reduces to the previously described conic case.

    Therefore, the area of a ball centered at a regular point can be upper bounded by the area of the ball without considering its intersection with the boundary (i.e., $4r^2$), plus the area of the largest ball centered at a conic point that remains entirely contained within the first ball.

    %Observe that balls centered on or near sides with Type 4 identifications follow the same procedure used for constructing balls centered at conic or regular points near a side with an infinite number of identifications (i.e., Type 3). The only difference in this case is that the resulting structure will always be a finite union of balls. Consequently, the estimates derived above remain valid for this case.

Finally, consider the case where the balls intersect more than one side of the multipolygon. Observe that this case essentially involves considering the balls centered on each side of the multipolygon since the radii decrease as parts of the balls appear due to the identifications of the sides. Therefore, the area of a ball can be upper-bounded by the sum of the areas of the balls with centers at conic points on each side of the multipolygon, plus the area of the first part of the ball. More precisely, we have:

Let $ B^i $ be the ball associated with the side $ e_i $, which means that $ B^i $ is a ball either centered at the conic point or at an accumulation point on the side $ e_i $. The area of the ball $ B(\tilde{p}, r) $ can then be upper-bounded by the following expression:

$$
B(\tilde{p}, r) \leq 4r^2 + 2n r^2  + \sum_{i=1}^n \max \left\{
\begin{array}{ll}
W_i & \text{if } e_i \text{ is a side with Type~$W$ identification}, \\
B{P}_i & \text{if } e_i \text{ is a side with basic segment pairings},
\end{array}
\right\}
$$
where

$$
W_i = \max \left\{ \operatorname{Area}(B^i(\tilde{p}_{\infty}, r)), \operatorname{Area}(B^i(\tilde{p}_c, r)) \mid \tilde{p}_\infty \in B^i(\tilde{p}_c, r), \right.
$$
$$
\left. \operatorname{Area}(B^i(\tilde{p}_c, r)) \mid \tilde{p}_\infty \notin B^i(p_c, r), \operatorname{Area}(\tilde{p}_{int}, r) \right\},
$$
and

$$
BP_i = \max \left\{ \operatorname{Area}(B^i(\tilde{p}_c, r)), \operatorname{Area}(B^i(\tilde{p}_{int} , r)) \right\}.
$$

Here, $ \tilde{p}_c $ refers to a conic point, and $ \tilde{p}_{int} $ refers to a regular point in the interior of a segment of pairings on a side of the multipolygon.

The term $ 4r^2 $ accounts for balls centered at interior points of the multipolygon, while $ 2n r^2 $ accounts for balls centered at the intersections of the sides, where $ n $ is the number of sides of the multipolygon.

The sum of the maximum values $ W_i $ and $ BP_i $ ensures that we are considering the largest possible areas of balls associated with sides having Type~$W$ identifications and basic segment pairings, respectively. In particular, for a side $ e_j $ with basic segment pairings, the area of a ball centered on this side is less than $ 2n_j r^2 $, where $ n_j $ is the number of conic points on this side.
%Specifically, $ W_i $ represents the maximum possible area for a side with Type~$W$ identification, accounting for whether the accumulation point $ \infty $ is inside or outside the ball. On the other hand, $ BP_i $ considers the maximum area between the balls associated with the accumulation point $ \infty $ and the interior points of the multipolygon $ p_i $.

%This formula provides a robust way to estimate the area of $ B(p, r) $, considering the different configurations of identifications on the sides of the multipolygon and offering a comprehensive approach for calculating the area of the balls in the context of multipolygons with Type~$W$ and basic segment pairing identifications.

To conclude, observe that in all cases, a lower bound for the area of the balls is given by  
$$
\operatorname{Area}(B(\tilde{p},r)) > 2r^2,
$$
since at the very least, the semi-ball must be considered - before the identifications are made.

\qed 

\subsubsection{Linear local contractibility}

To establish the linear local contractibility of the class of surfaces $\mathcal{L}$, we require the following result. This lemma was kindly provided by Mario Bonk, and a proof is included in Appendix A.

\begin{lemma}\label{defret1}

 Let $X,Y\subseteq S^2$ be subsets such that $Y\subseteq X$. If $Y$ is a deformation retract of $X$ and $Y^c$ is connected, then $X^c$ is connected. 

   \end{lemma}

\emph{Proof of Theorem \ref{teo_principal}(b).} 
We now aim to show that the surface $ \mathcal{L}_\mathcal{P} $ is $\lambda$-linearly locally connected ($\lambda$-LLC). In this setting, the $\lambda$-LLC condition is equivalent to proving that the space is linearly locally contractible.

%Throughout this proof, any reference to a ball refers to its preimage under the projection map and to a conic point or {conic point of angle $ k\pi $} (with $ k \in \mathbb{N} $), refers to one of its $ k $ preimages at the boundary of $ P $.

Let $ B(a, r) $ denote the a ball in $ \mathcal{L}_{\mathcal{P}} $, centered at a point $ a $ with radius $ r $. If we choose $ \lambda = 4K $, where $ K $ is the constant defined in the proof of Theorem~\ref{teo_principal}(a), then the preimage of such a ball does not intersect two non-adjacent sides of the multipolygon before the boundary identifications are applied. 

To verify the $\lambda$-LLC$_1$ condition, note that $ \mathcal{L}_\mathcal{P} $ is a compact length space. Thus, for any two points $ x, y \in B(a, r) $, we can construct a continuum joining them by taking the union of geodesic segments from $ x $ to $ a $ and from $ y $ to $ a $.

%univ.thm

To prove the $\lambda$-LLC$_2$ condition, further analysis is required. Initially, consider the following: let $ B(a, r) $ be the preimage of a ball with center $ a $ and radius $ r $, and let $ x $ and $ y $ be points on the surface that lie outside $ B(a, r) $. If the geodesic joining the points $ x $ and $ y $ on the surface does not pass through the ball $ B(a, r/\lambda) $, then the continuum $ E $ can be chosen as the geodesic itself.

If the geodesic passes through the ball $ B(a, r/\lambda) $, since there exist paths connecting any two points on the surface, $ E $ can be taken as a path that avoids the ball. For example, a geodesic that avoids the ball can be constructed as follows: consider segments emanating from $ x $ and $ y $ that reach the closest points on the boundary of the ball, and then connect these new points using a path along the boundary of the ball. To ensure the existence of such a path, it is necessary that the complement of the ball is connected.

It is important to examine what happens to balls that do not intersect opposite sides. To establish this condition, Lemma \ref{defret1} will be required.

There are three cases to consider: 

\begin{enumerate}[1)]
    \item Balls centered at conic points.
    \item Balls centered at the accumulation point.
    \item Balls centered at regular points.
\end{enumerate}

For balls centered at either conic points or regular points, with radii small enough to exclude any other conic point, a deformation retraction to a point can be constructed. The procedure consists of first deforming the ball toward a side of the polygon via linear homotopies. Once a segment on the boundary is reached, a further deformation is performed along this segment, ultimately contracting the set to a point. In contrast, when a ball contains a conic point and intersects the sides of the multipolygon, a deformation retraction can still be constructed - this time onto a union of segments along the polygon’s boundary that intersect the ball.

For balls centered at the accumulation point, a deformation retraction onto a union of segments along the sides of the multipolygon can also be constructed using linear deformations.

In summary, we have established that, depending on the location of the ball, its deformation retract may be a point, a single segment, or a union of segments along the sides of the multipolygon.

As a consequence, the complement of any such ball is connected, which allows us to construct a continuum connecting any two points outside the ball via a path that entirely avoids its interior. Therefore, by Lemma~\ref{defret1} and Theorem~\ref{univ.thm}, the surface $ \mathcal{L}_\mathcal{P} $ is linearly locally connected, and hence, linearly locally contractible.

\qed

As a corollary of items (a) and (b) in Theorem~\ref{teo_principal}, we obtain the following fundamental result:  
the class of surfaces $ \mathcal{L}^* $ is quasisymmetrically equivalent to the standard 2-sphere.

This follows from the Bonk-Kleiner Theorem, which asserts that any compact, Ahlfors 2-regular, and linearly locally connected metric space homeomorphic to $ {S}^2 $ is quasisymmetrically equivalent to the standard 2-sphere.

%\textit{Remark.} It is possible to prove that any regular $ n $-gon is quasisymmetrically equivalent to the topological 2-sphere. The proof established for the square, proving the conditions of Ahlfors 2-regularity and linear local connectivity, extends inductively to any regular $ n $-gon.

%Basically, by multiplying the estimates by $ n $, and considering that balls with centers on other sides may appear, it is possible to establish Ahlfors 2-regularity. Furthermore, similar deformation retracts of balls can be constructed to verify the linearly locally connected condition.

\section{Tight horseshoe}

The \emph{tight horseshoe} is a transformation defined on a topological sphere that replicates the dynamics of Smale's classical horseshoe map. However, unlike Smale’s original construction, the tight horseshoe does not contain points with trivial dynamics. In this case, the non-wandering set of the transformation is the entire surface.
This characteristic renders the system minimal in a dynamical sense, as every point exhibits nontrivial behavior under iteration.

 In this section, we show that the tight horseshoe provides an example of a paper space that fails to satisfy the Ahlfors 2-regularity condition at a simple point. 
 Consequently, the Bonk–Kleiner theorem cannot be applied to conclude that the tight horseshoe sphere is quasisymmetrically equivalent to the standard 2-sphere.

Let $\{a_i\}_{i \in \mathbb{N}}$ be a sequence of positive real numbers with $a_0 = 1/2$ and $\sum_{i=1}^{\infty} a_i = 1/2$. Consider $P$ as a unit square and define $\hat{P} = \{\langle \alpha_i, \alpha_i' \rangle\}_{i \in \mathbb{Z}}$ according to the following construction - we proceed as in the standard Type~$ W $ identification, except that all segments $ \beta_i $ are omitted (i.e., $ b_i = 0 $ for all $ i $):

The segment $\alpha_0$ has length $a_0$ and starts at the upper-right vertex of $P$. The segment $\alpha_0'$ also has length $a_0$ and begins at the endpoint of $\alpha_0$, folding the top side of $P$ in half. Next, $\alpha_1$ has length $a_1$ and starts at the endpoint of $\alpha_0'$, followed by $\alpha_1'$ of the same length, starting at the endpoint of $\alpha_1$. This process continues, forming a sequence of folds from top to bottom that covers the left side of $P$, except for its lower vertex.

A similar process is applied to the right side of $P$. The segment $\alpha_{-1}'$ has length $a_0$ and ends at the upper-right vertex of $P$, while $\alpha_{-1}$, also of length $a_0$, starts at the initial point of $\alpha_{-1}'$, folding the right side of $P$ in half. Subsequently, $\alpha_{-2}'$ has length $a_1$ and ends at the initial point of $\alpha_{-1}$, while $\alpha_{-2}$, also of length $a_1$, starts at the initial point of $\alpha_{-2}'$. This pattern continues, forming folds from right to left that cover the lower side of $P$, except for its leftmost endpoint.

 \begin{figure}[!h]
           \centering
           \includegraphics[scale=1.1]{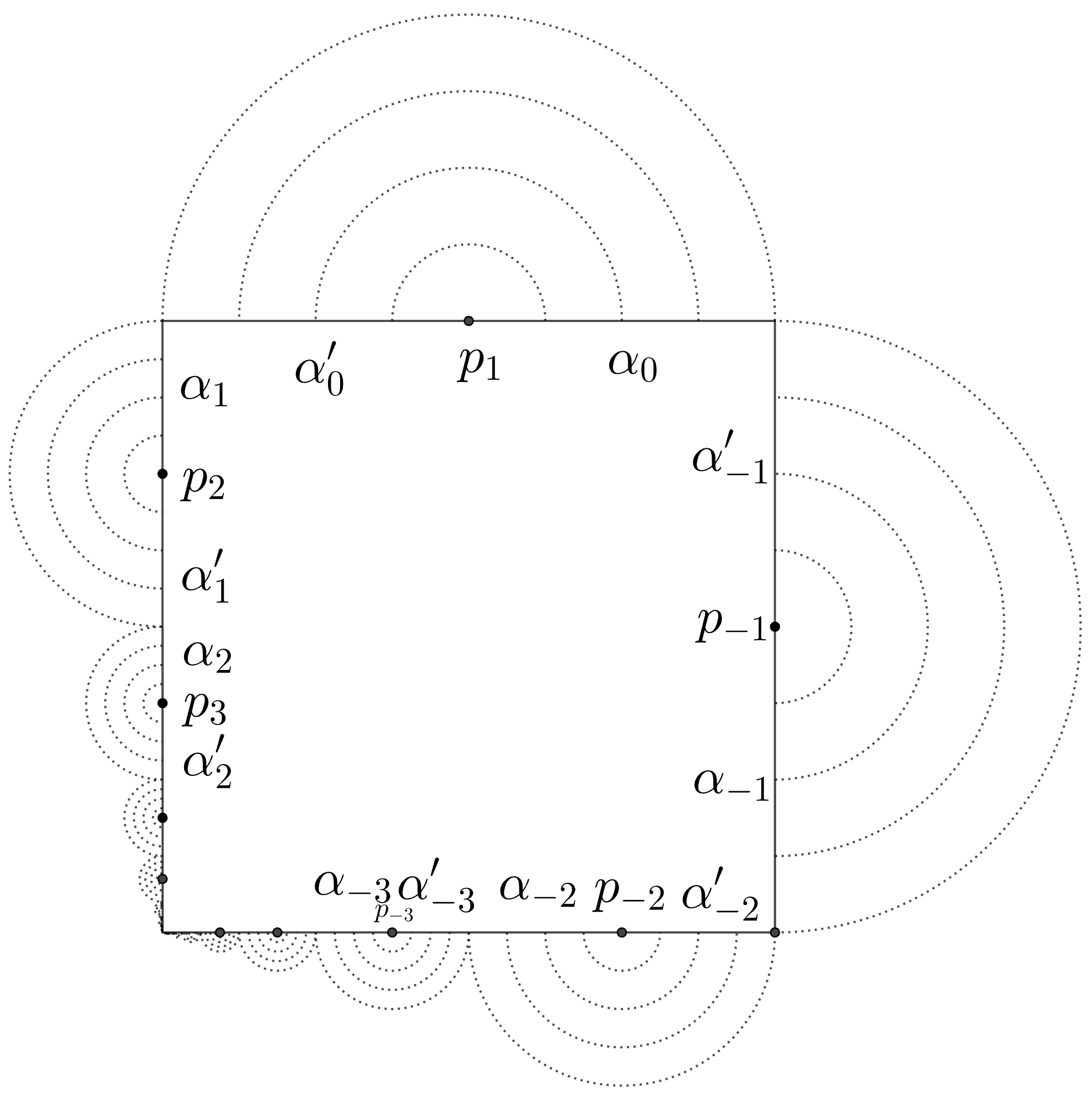}
           \caption{The square $P$ with all the identifications and fold points.}
           \label{identif_vert3}
       \end{figure}
 
The tight horseshoe is not Ahlfors 2-regular, and the issue arises at the accumulation point due to the fact that all endpoints of the folds are identified. Moreover, as the radius of the ball decreases, additional balls will appear at these points.

In this case, a ball centered at the accumulation point is given by  

\begin{align*}
    B(\tilde{p}_{\infty}, r) = & \ 4 \text{ one-quarter balls of radius } r \cup (2n - 2) \text{ semiballs of radius } r \ \cup \\ &\bigcup_{n+1}^{\infty} \text{ pieces of balls of radius } r,
\end{align*}
where $n$ is such that $\frac{1}{2^n} < r$.

Therefore,
 $$ \text{Area}(\Bar{B}(\tilde{p}_{\infty},r))= \left[(\pi-2) r^2 +\left(2\log \frac{1}{r}\right)\left(\frac{\pi}{2} r^2\right) + \text{ something}\right]. $$

Despite this, \cite{de2005extensions} and \cite{de2012paper} show, using different techniques, that the tight horseshoe sphere has a well-defined conformal structure and is therefore conformally equivalent to~$\overline{\mathbb{C}}$.

\section*{Appendix A}\label{appendix}

In this appendix, we discuss criteria for determining when an open subset $ U $ of a connected orientable manifold $ X $ is connected. Under these assumptions, it is standard that $ U $ is connected if and only if it is path-connected. The main reference for this section is Chapter 3 in \cite{hatcher}. Furthermore, the proofs of the following results were shown to me by Mario Bonk.

This topological property can be detected homologically via the singular homology group $ H_0(U) $, which encodes the number of path-connected components of $ U $. More precisely, $ H_0(U) $ is isomorphic to a direct sum $ \bigoplus_{i \in I} \mathbb{Z} $, where the index set $ I $ has one element for each path component. Thus, $ U $ is connected if and only if $ H_0(U) \cong \mathbb{Z} $.

It is often convenient to work instead with the reduced homology group $ \widetilde{H}_0(U) $, obtained by quotienting out one of the $ \mathbb{Z} $-summands. In this setting, $ U $ is connected if and only if $ \widetilde{H}_0(U) = 0 $. Since $ X $ is assumed connected, we also have $ \widetilde{H}_0(X) = 0 $.

Using the long exact sequence of the pair $ (X, U) $, we obtain:
$$
\cdots \to H_1(X) \to H_1(X,U) \to \widetilde{H}_0(U) \to \widetilde{H}_0(X) = 0.
$$
From exactness, it follows that the connecting map $ H_1(X,U) \to \widetilde{H}_0(U) $ is always surjective. Furthermore, if $ H_1(X) = 0 $ (e.g., when $ X $ is simply connected), then the map is an isomorphism. This yields the following criterion.

\begin{lemma}
\label{lemma:A1}
Suppose that $ H_1(X) = 0 $. Then $ U $ is connected if and only if $ {H_1(X,U) = 0}. $
\end{lemma} 

\begin{proof}
The map $ H_1(X,U) \to \widetilde{H}_0(U) $ is surjective in general, and it is an isomorphism when $ H_1(X) = 0 $. Since $ \widetilde{H}_0(U) = 0 $ if and only if $ U $ is connected, the claim follows.
\end{proof}

Now assume $ X $ is a compact connected orientable $ n $-manifold. By Alexander's duality, we have an isomorphism
$$
H_1(X,U) \cong \check{H}^{n-1}(X \setminus U),
$$
where $ \check{H}^{n-1} $ denotes Čech (or Alexander–Spanier) cohomology. This leads to the following.

\begin{lemma}
Suppose $ X $ is a compact connected orientable $ n $-manifold, $ n \geqslant 2 $, with $ H_{1}(X) = 0 $. Suppose $ K, L \subseteq X $ are compact sets that are homotopy equivalent. Then $ X \setminus K $ is connected if and only if $ X \setminus L $ is connected
\end{lemma}

\begin{proof}
Suppose $ X \setminus K $ is non-empty and connected. Then:
$$
0 = \widetilde{H}_0(X \setminus K) \cong H_1(X, X \setminus K) \cong \check{H}^{n-1}(K)
\cong \check{H}^{n-1}(L) \cong H_1(X, X \setminus L) \cong \widetilde{H}_0(X \setminus L).
$$
Thus, $ X \setminus L $ is also connected.
\end{proof}

We note that the assumption $ H_1(X) = 0 $ is essential. For example, if $ X $ is a 2-torus and $ K, L \subset X $ are topological circles, the result fails if one of them bounds a disk while the other is homotopically non-trivial.

To refine the previous lemma, we now recall the notion of deformation retracts. Let $ K \subset X $ be a subset. We say that $ K $ is a deformation retract of $ X $ if there exists a homotopy $ H : X \times [0,1] \to X $ such that $ H_0 = \mathrm{id}_X $, $ H_1(X) \subset K $, and $ H_t(x) = x $ for all $ x \in K $ and $ t \in [0,1] $. In this case, $ X $ and $ K $ are homotopy equivalent.

\begin{lemma}
Let $ X $ be a topological space and suppose $ K \subseteq X $ is a deformation retract of $ X $. Then $ K $ and $ X $ are homotopy equivalent
\end{lemma}

\begin{proof}
Let $ H: X \times [0,1] \rightarrow X $ be a homotopy as in the definition of a deformation retract of $ X $ onto $ K $. Then the map $ H_{1}: X \rightarrow X $ takes only values in $ K $. So we can consider it as a map $ g: X \rightarrow K $. In particular, $ g(x) = H_{1}(x) $ for all $ x \in X $ and $ g(x) = H_{1}(x) = x $ for all $ x \in K \subseteq X $. Let $ f: K \rightarrow X $ be the inclusion map, that is, $ f(x) = x \in X $ for all $ x \in K $.

Then for $ x \in K $ we have $ (g \circ f)(x) = g(x) = H_{1}(x) = x $, and so $ g \circ f = \mathrm{id}_{K} $. On the other hand, $ f \circ g = H_{1} $ by definition of $ g $, and so $ f \circ g = H_{1} \simeq H_{0} = \mathrm{id}_{X} $. The statement follows.
\end{proof}

Applying these results to the case of spheres, we obtain the following.

\begin{theorem}
Let $ S^{n} $, $ n \geqslant 2 $, be the $ n $-dimensional sphere, and let $ K \subseteq L $ be compact subsets of $ S^{n} $. If $ K \subseteq L $ is a deformation retract of $ L $, then $ S^{n} \setminus L $ is connected if and only if $ S^{n} \setminus K $ is connected
\end{theorem}

\begin{proof}
This follows from our previous considerations and the fact that $ S^{n} $ is simply connected for $ n \geqslant 2 $, and so $ H_1(S^n) = 0 $.
\end{proof}

\subsection*{Acknowledgments}

The author would like to acknowledge the financial support received from FAPESP (Grant numbers 2020/06978-6 and 2022/05984-8) and CAPES (Process number 88887.103459/2025-00). The author also thanks André de Carvalho and Mario Bonk for their helpful and insightful comments.

\bibliography{bib_luciana}

\bibliographystyle{alpha}
%\bibliography{references}  %%% Uncomment this line and comment out the ``thebibliography'' section below to use the external .bib file (using bibtex) .





% \end{thebibliography}

\end{document}